\documentclass[pdflatex,sn-mathphys-num]{sn-jnl}

\usepackage[T1]{fontenc}
\usepackage[utf8]{inputenc}
\usepackage{amsmath,amssymb,amsfonts}
\usepackage{bm}
\usepackage{enumitem}
\usepackage{booktabs}
\usepackage{xcolor}

\numberwithin{equation}{section}
\allowdisplaybreaks

\theoremstyle{thmstyleone}
\newtheorem{theorem}{Theorem}[section]
\newtheorem{proposition}[theorem]{Proposition}
\newtheorem{lemma}[theorem]{Lemma}
\newtheorem{corollary}[theorem]{Corollary}

\theoremstyle{thmstyletwo}
\newtheorem{remark}[theorem]{Remark}

\theoremstyle{thmstylethree}

\newenvironment{acks}[1][]{\par\section*{#1}}{\par}

\begin{document}

\title[Joint laws of maximum drawdown and drawup]{Joint Laws of Maximum Drawdown and Maximum Drawup for Spectrally Negative L\'{e}vy Processes}

\author*[1]{\fnm{Vardar Acar} \sur{Ceren}}\email{cvardar@metu.edu.tr}
\author[2]{\fnm{Akdo\u{g}an} \sur{Emre}}\email{emre.akdogan@tubitak.gov.tr}

\affil*[1]{\orgname{Department of Statistics and Institute of Applied Mathematics, Middle East Technical University}, \orgaddress{\city{Ankara}, \country{T\"urkiye}}}
\affil[2]{\orgname{T\"UB\.ITAK, Turkey Defense Industries and Development Instıtute SAGE}, \orgaddress{\city{Ankara}, \country{T\"urkiye}}}

\abstract{Let $X$ be a spectrally negative L\'{e}vy process observed up to an independent exponential time $T$ with parameter $\gamma>0$. We study the joint law of the maximum drawdown and the maximum drawup of $X$ on $[0,T]$. The path is decomposed according to the two possible  orderings of its  infimum and supremum. Conditional on the values of these extrema and on their ordering, the resulting pre-, intermediate, and post- components are independent. We identify their laws as Doob $h$-transforms of killed spectrally negative L\'{e}vy processes and express the corresponding distribution functions explicitly in terms of the $\gamma$-scale functions $W^{(\gamma)}$ and $Z^{(\gamma)}$ and their derivatives. Combining the conditional laws with the joint densities of the  extrema yields integral representations of the joint distribution and moments.}



\maketitle

\section{Introduction}\label{sec:introduction}

Maximum drawdown and maximum drawup measures, respectively, the largest decline and the largest realized gain along an investment, trading account, or asset over a specific time period.  Their distributions are relevant in risk measurement, investment policies, portfolio constraints, drawdown-linked contracts, and path-dependent derivatives; see, among others, \cite{Pospisil2010,Baurdoux2017, Hadjiliadis2006,Zhang2010,cvitanic1995portfolio,grossman1993optimal, Chekhlov2005,Magdon-Ismail2004,Leal2005,Carr2011,Vecer2007,Huang2013,Rossello2021,Sornette2003,Vecer2006,Avram2004,Mijatovic2012,Landriault2015,Landriault2016,Landriault2017}). The distributional properties of maximum drawdowns are studied in (\cite{Douady2000,Magdon-Ismail2004,Salminen2007,VardarAcar2017,VardarAcar2021,Salminen2020,salminen2025drawdowns,ZHANG2023104669,10.1214/ECP.v20-3945,risks7040105}.  In the spectrally negative L\'{e}vy setting, these functionals are naturally connected with fluctuation identities and first-passage problems, for which the scale functions $W^{(q)}$ and $Z^{(q)}$ provide the basic analytic tools; see \cite{Bertoin1993, Bertoin1996,kyprianou2006introductory,rivero2013,avram2020}.

For Brownian motion and one-dimensional diffusions, maximum increase and maximum decrease were studied in \cite{Salminen2007,Salminen2020}, while drawdown and drawup identities for spectrally negative L\'{e}vy processes were obtained in \cite{VardarAcar2017,VardarAcar2021,VardarAcar2026}.  Path decompositions at extrema, together with applications to maximum drawdown and drawdown duration, were developed in \cite{VardarAcar2021}.  A remaining ingredient in that program was the conditional law of the pre-infimum component of the entire process.  The identification of this component in \cite{VardarAcar2026} makes it possible to complete the decomposition when the infimum is attained before the supremum and, in the present paper, to combine the analysis with the case supremum is attained before the infimum into finding the joint distribution of maximum drawdown and maximum drawup together with various other joint and marginal distributions related to these functionals.

Let $X$ be a spectrally negative L\'{e}vy process observed up to an independent exponential time $T$ with rate $\gamma>0$.  Write $I_T$ and $S_T$ for the infimum and supremum on $[0,T]$, and $H_I$ and $H_S$ for their attainment times.  Conditional on $I_T=a<0<S_T=b$, the path is decomposed according to the ordering of these two times.  On $\{H_S<H_I\}$ the relevant pieces are the pre-supremum, the intermediate path from the supremum to the infimum, and the post-infimum component.  On $\{H_I<H_S\}$ the corresponding pieces are the pre-infimum, the intermediate path from the infimum to the supremum, and the post-infimum path.  Conditional on these extrema values and their orderings, the three components are independent in each case.

The purpose of this paper is to express the laws of these components and to find distribution of maximum drawdowns and maximum drawups in each component directly through $W^{(\gamma)}$, $Z^{(\gamma)}$ scale functions, and their derivatives, then to assemble these conditional pieces into finding the joint distribution of $(M_T^-,M_T^+).$ Along the way of finding this joint distribution various other joint distributions are also proved such as the joint distribution of the infimum and the supremum together with the two orderings of these extrema, the maximum drawdown together with the case infimum occurs before the supremum and the maximum drawup where the supremum occurs before the infimum. And subsequently, the mixed moments and correlation among maximum drawdown and maximum drawup are calculated. All these novel findings are applicable and relevant in studies concerning risk assessment, portfolio constraints, drawdown and drawup-linked contracts, and path-dependent derivatives. The results obtained in terms of the scale function formulas and they all reduce to the classical decomposition in the Brownian motion case given in study \cite{Salminen2007}. 

The paper is organized as follows.  Section~\ref{sec:preliminaries} contains the fluctuation identities and notation.  Section~\ref{sec:decsup} gives the decomposition at the supremum and the associated pre- and post-supremum laws as well as the distributions of the maximum drawdowns and maximum drawups in each component. Similarly, Section~\ref{sec:path-decomposition-at-infimum} gives the decomposition at the infimum and corresponding path properties. Section~\ref{sec:HS-before-HI} treats the decomposition under $H_S<H_I$ and characterizes each path and finds the maximum drawup distributions on each of its three independent components.  Section~\ref{sec:HI-before-HS} recalls the decomposition at the the reverse ordering $H_I<H_S$, and the corresponding maximum-drawdown laws.  Section~\ref{sec:joint-distributions-of-maximum-drawdowns-and-maximum-drawups} combines the two orderings to obtain the joint and marginal distributions.  Section~\ref{sec:correlation} derives the moments and the correlation, and Section~\ref{sec:conclusion} concludes the study.

\section{Preliminaries}\label{sec:preliminaries}
\label{sec:mathematical-foundations}

Let $X=(X_t)_{t\ge0}$ be a spectrally negative L\'{e}vy process with unbounded variation on a complete filtered probability space $(\Omega,\mathcal{F},(\mathcal{F}_t)_{t\ge0},\mathbb{P})$.  Its Laplace exponent $\psi$ is defined by
\begin{equation}\label{eq:laplace-exponent}
\mathbb{E}_x\big[e^{\lambda(X_t-X_0)}\big]=e^{t\psi(\lambda)},\qquad \lambda\ge0,
\nonumber\end{equation}
and has the L\'{e}vy--Khintchine representation
\[
\psi(\lambda)=\frac{\sigma^2}{2}\lambda^2+\mu\lambda
+\int_{(0,\infty)}\bigl(e^{-\lambda y}-1+\lambda y\mathbf{1}_{\{y<1\}}\bigr)\Pi(\mathrm{d} y),
\qquad \lambda\ge0,
\]
where \(\sigma\ge0\), \(\mu\in\mathbb R\), and \(\Pi\) is the L\'{e}vy measure of \(-X\), satisfying \(\int_{(0,\infty)}(1\wedge y^2)\Pi(\mathrm{d} y)<\infty\).

For $q\ge0$, let
\[
\Phi(q)=\sup\{\lambda\ge0:\psi(\lambda)=q\}.
\]
The $q$-scale function $W^{(q)}:\mathbb{R}\to[0,\infty)$ is characterized by $W^{(q)}(x)=0$ for $x<0$ and
\begin{equation}\label{eq:scale-def}
\int_0^\infty e^{-\lambda x}W^{(q)}(x)\,\mathrm{d} x=\frac{1}{\psi(\lambda)-q},
\qquad \lambda>\Phi(q).\nonumber
\end{equation}
The associated second $q$-scale function is found as
\begin{equation}\label{eq:Z-def}
Z^{(q)}(x)=1+q\int_0^x W^{(q)}(y)\,\mathrm{d} y,
\qquad x\ge0.\nonumber
\end{equation}

Since $X$ has unbounded variation, the origin is regular for both half-lines and
$W^{(q)}$ is continuously differentiable on $(0,\infty)$
\cite[Lem.~2.4]{rivero2013}. Because several results involve a second
derivative of the scale function, we impose the following additional regularity
assumption. For the fixed rate $q>0$, the scale function $W^{(q)}$ belongs to $C^2(0,\infty)$. Whenever a boundary value is used, the corresponding
right-hand derivative is understood. This assumption holds, for example, when $X$ has a nonzero Gaussian component; it may also hold under suitable smoothness conditions on the L\'{e}vy measure. It is stated explicitly because unbounded variation alone does not ensure the existence of
$W^{(\gamma)\prime\prime}$.

For later verification, consider standard Brownian motion, where
\(\psi(\lambda)=\lambda^2/2\). Its \(\gamma\)-scale functions and the
derivatives used below are
\begin{align}
\Phi(\gamma)&=\sqrt{2\gamma},~~ W^{(\gamma)}(x)=\frac{2}{\sqrt{2\gamma}}
\sinh\!\left(\sqrt{2\gamma}\,x\right),
~~W^{(\gamma)\prime}(x)=2\cosh\!\left(\sqrt{2\gamma}\,x\right),\nonumber\\
W^{(\gamma)\prime\prime}(x)&=2\sqrt{2\gamma}
\sinh\!\left(\sqrt{2\gamma}\,x\right),
~~W^{(\gamma)\prime\prime\prime}(x)=4\gamma
\cosh\!\left(\sqrt{2\gamma}\,x\right),
~~ Z^{(\gamma)}(x)=\cosh\!\left(\sqrt{2\gamma}\,x\right),\nonumber\\
Z^{(\gamma)\prime}(x)&=\sqrt{2\gamma}
\sinh\!\left(\sqrt{2\gamma}\,x\right)=\gamma W^{(\gamma)}(x),~~
Z^{(\gamma)\prime\prime}(x)=2\gamma
\cosh\!\left(\sqrt{2\gamma}\,x\right)=2\gamma Z^{(\gamma)}(x),
x\geq0.
\label{eq:brownian-scale-functions}
\end{align}
In particular for standard Brownian motion, \(W^{(\gamma)}(0)=Z^{(\gamma)\prime}(0)=0\),
\(W^{(\gamma)\prime}(0+)=2\), and
\(W_{+}^{(\gamma)\prime}(x)=W^{(\gamma)\prime}(x)\) for \(x>0\).

We now consider the first passage times defined by
\[
\tau_c^+:=\inf\{t\ge0:X_t>c\},\qquad
\tau_c^-:=\inf\{t\ge0:X_t<c\}.
\]
For $a<x<b$, the standard two-sided exit identities are
\begin{align}
\mathbb{E}_x\left[e^{-q\tau_b^+};\tau_b^+<\tau_a^-\right]
&=\frac{W^{(q)}(x-a)}{W^{(q)}(b-a)},\label{eq:two-sided-up}\\
\mathbb{E}_x\left[e^{-q\tau_a^-};\tau_a^-<\tau_b^+\right]
&=Z^{(q)}(x-a)-Z^{(q)}(b-a)\frac{W^{(q)}(x-a)}{W^{(q)}(b-a)}.\label{eq:two-sided-down}
\end{align}
We also use the one-sided identity
\begin{equation}\label{eq:one-sided-down}
\mathbb{E}_x\left[e^{-q\tau_a^-};\tau_a^-<\infty\right]
=Z^{(q)}(x-a)-\frac{q}{\Phi(q)}W^{(q)}(x-a),\qquad x>a.\nonumber
\end{equation}
  In the following sections, the killed processes are formulated through passage times $\tau_c^\pm$ except for the pre-infimum process where the ordinary hitting time   $\tau_c:=\inf\{t\ge0:X_t=c\}$ is required. On the other hand, $\tau_c^{+} =\tau_c$ because $X$ has no positive jumps and the process creeps upwards.
  
Let $T\sim \operatorname{Exp}(\gamma)$, $\gamma>0$, be independent of $X$.  We define
\[
S_T:=\sup_{0\le t\le T}X_t,\qquad I_T:=\inf_{0\le t\le T}X_t,
\]
and denote by $H_S$ and $H_I$ the times at which these extrema are attained, respectively, that is, 
\[
H_S=\inf\{t < T: X_{t}=S_T\}, \qquad H_I=\inf\{t < T: X_t=I_T\} \; .
\]
 Since $X$ has no positive jumps, the  supremum is attained continuously, and also 0 is regular for $(0, \infty)$.

It follows from  \cite[Thm.3.1]{Millar1977} that 
\begin{equation}
X_{H_I} = I_T  \label{eq:-infimum-attainment}
\end{equation}
almost surely under $H_I<T$.  
    This can happen in two ways: the process $X$ either jumps into $I_T$ or $X$ is continuous at  $I_T$. If 0 is regular also for $(-\infty,0)$, which happens when $X$ is of unbounded variation, then $X$ is continuous at $I_T$ (\cite{Millar1977}, pg.370, Remark after Prop.2.4).

Statements involving conditioning on $I_T=a$ or on $(I_T,S_T)=(a,b)$ are
understood under fixed versions of the corresponding regular conditional laws.
For $a<0<b$, we use the notation
\begin{align}
\mathbb{P}_{a,b}^{<}(\,\boldsymbol\cdot\,)
&:=\mathbb{P}_0(\,\boldsymbol\cdot\mid H_I<H_S,I_T=a,S_T=b),
\label{eq:conditional-law-HIHS}\\
\mathbb{P}_{a,b}^{>}(\,\boldsymbol\cdot\,)
&:=\mathbb{P}_0(\,\boldsymbol\cdot\mid H_S<H_I,I_T=a,S_T=b).
\label{eq:conditional-law-HSHI}
\end{align}
Thus, the superscript records the  ordering of the  extrema:
``$<$'' corresponds to $H_I<H_S$, whereas ``$>$'' corresponds to $H_S<H_I$.

Now, for $t\geq0$,
$S_t:=\sup_{0\leq s\leq t}X_s,~I_t:=\inf_{0\leq s\leq t}X_s.$ And,
let the first passage time at level $d$ of the drawdown process be defined by $\sigma_d=\inf \left\{t \geq 0: Y_t>d\right\}$ where $Y_t:=S_t-X_t,~t\geq 0$ also let the first passage time at level $u$ of the drawup process be defined by 
$\hat{\sigma}_d=\inf \left\{t \geq 0: \hat{Y}_t>d\right\}$
where $\hat{Y}_t:=X_t-I_t,~t\geq0.$

The maximum drawdown and maximum drawup on $[0,T]$ are defined directly by
\begin{equation}\label{eq:mdd-mdu-def}
M_T^-:=\sup_{0\le u\le v\le T}(X_u-X_v),\qquad
M_T^+:=\sup_{0\le u\le v\le T}(X_v-X_u).
\end{equation} And, more generally, for \(0\le p\le q\),
$
 M_{p,q}^-=\sup_{p\le u\le v\le q}(X_u-X_v),\qquad
 M_{p,q}^+=\sup_{p\le u\le v\le q}(X_v-X_u).
$

A positive excessive function $h$ for a killed Markov process defines the Doob $h$-transform by
\begin{equation*}\label{eq:h-transform}
P_t^h(x,\mathrm{d} y)=\frac{h(y)}{h(x)}\mathbb{P}_x\{X_t\in\mathrm{d} y,\ t<\zeta\},
\end{equation*}
where $\zeta$ is the killing time.  Positivity and excessivity of the functions $h$  follow from their representation as conditional  entrance probabilities for the corresponding killed semigroups.

To obtain the joint law of \((M_T^-,M_T^+)\) the methodology of reducing the path to conditional laws on independent path pieces will be used.  The reduction starts from the disjoint decomposition
\begin{equation}\label{eq:ordering-split}
 \mathbb{P}\{M_T^-<\alpha,M_T^+<\beta\}
 =
 \mathbb{P}\{M_T^-<\alpha,M_T^+<\beta,H_I<H_S\}
 +
 \mathbb{P}\{M_T^-<\alpha,M_T^+<\beta,H_S<H_I\}.
\end{equation}
The two terms in this sum require different analyses because the lack of positive jumps in spectrally negative Levy process makes the pre- and post-extremal values boundary conditions asymmetric.

On the event \(H_I<H_S\), conditional on \(I_T=a\) and \(S_T=b\), the path is split into
\[
 [0,H_I],\qquad [H_I,H_S],\qquad [H_S,T].
\]
The full-path drawup is the range \(S_T-I_T=b-a\) under the given condition, while the full paths drawdown is the maximum of the three segmentwise drawdowns \(M_{0,H_I}^-\), \(M_{H_I,H_S}^-\) and \(M_{H_S,T}^-\).  Hence, for bounded measurable \(\varphi\),
\begin{align}\label{eq:HIHS-reduction}
&\mathbb{E}\!\left[\varphi(M_T^+);\,M_T^-<\alpha,\,H_I<H_S\right]\nonumber\\
&\quad =
\mathbb{E}\!\left[\varphi(S_T-I_T);\,
       M_{0,H_I}^-<\alpha,\,
       M_{H_I,H_S}^-<\alpha,\,
       M_{H_S,T}^-<\alpha,\,
       H_I<H_S\right].
\end{align}
Given \((I_T,S_T)\) and \(H_I<H_S\), the three path pieces are independent; consequently the conditional expectation in \eqref{eq:HIHS-reduction} factorizes into a product of three scale-function expressions.

On the event \(H_S<H_I\), the path is split into
\[
 [0,H_S],\qquad [H_S,H_I],\qquad [H_I,T].
\]
Now the full-path drawdown is the range \(b-a\), while the drawup is the maximum of the pre-supremum, intermediate and post-infimum components drawdowns. The analogous factorization gives
\[
 \mathbb{E}\!\left[\varphi(M_T^-);\,M_T^+<\beta,\,H_S<H_I\right],
\]
again as an integral of a product of conditional path pieces laws against the joint law of \((I_T,S_T)\) on the ordering \(H_S<H_I\) event.  Substituting indicator functions for \(\varphi\) in the two reductions and treating the deterministic range constraint \(b-a\) gives the case-by-case formulae of the joint distribution of \((M_T^-,M_T^+)\).

Based on the conditional characterizations of the independent components of the path, the extrema densities under the two orderings are also derived and then used to remove the conditioning in the drawdown-drawup formulae.

\section{Path decomposition at the supremum} \label{sec:decsup}
We begin decomposing the process by conditioning only on the value of the supremum.  This decomposition will later be used for finding the law of the additional conditioning $\mathbb{P}_0( H_S<H_I,I_T=a,S_T=b)$ required on the process which is also on the infimum and the time ordering of these extrema values. 

In the following Theorem~\ref{lemma_pre_post_sup_dist}, we recall and present the distributions of the pre-supremum process \\$\left\{X_{u}: 0 \leq u \leq H_{S}\right\}$ and post-supremum process $\left\{X_{H_{S}+u}: 0 \leq u \leq\right.$ $\left.T-H_{S}\right\}$ and maximum drawdown distributions for each segment obtained in Lemma 1 of \cite{VardarAcar2021}. In addition to these results, we prove the distributions of the infimum before and after the time of the supremum together with the maximum drawup distributions of each segment. This result isolates the pre-supremum and the post-supremum components before additional conditioning of the ordering of two extrema is imposed. 
\begin{theorem}\label{lemma_pre_post_sup_dist}
Conditionally on $S_{T}=b$, it holds that:
	
	a) The pre-$H_S$ process and the post-$H_S$ process are independent.
	
	b) The law of pre-$H_S$ process is an \(h\)-transform of the law of the spectrally negative L\'{e}vy process killed at $T \wedge \tau_{b}^{+}$, with
	$$
	h(x)=e^{-\Phi(\gamma)(b-x)}, \quad x \leq b.
	$$
		More precisely, it is a spectrally negative L\'{e}vy process killed when it hits $b$, and governed by an Esscher exponential change of measure $\mathbb{P}^{\Phi(\gamma)}$ of the original measure, with	
	\begin{equation}\label{eqn9}
		\left.\frac{d \mathbb{P}^{\Phi(\gamma)}}{d \mathbb{P}}\right|_{\mathcal{F}_{t}}=e^{\Phi(\gamma) X_{t}-\gamma t} 
	\end{equation}
	and with Laplace exponent
	$$
	\widetilde{\psi}(\lambda)=\psi_{\Phi(\gamma)}(\lambda):=\psi(\lambda+\Phi(\gamma))-\gamma, \quad \lambda \geq-\Phi(\gamma).
	$$

Moreover, the distribution of the infimum of the pre-supremum process, $I_{0,H_{S}}$, when $S_{T}=b$ is given by
	\begin{equation}
	\mathbb{P}_0\left(I_{H_{S}} > a \mid S_{T}=b\right)=\frac{e^{\Phi (\gamma)(b)}W^{ (\gamma)}(-a)}{W^{(\gamma)}(-a+b)}
	\end{equation}
	for $b>a$ and for $u\ge b$
    \begin{align*}
\mathbb{P}\left\{M_{0, H_{S}}^{+}<u  \mid S_{T}=b \right\} =e^{\Phi(\gamma)(b)}\frac{W^{(\gamma)}(u-b)}{W^{(\gamma)}(u)}
\end{align*}
 also for $d\geq 0$
 \begin{align*}
\mathbb{P}\left\{M_{0, H_S}^{-}<d \mid S_T=b\right\}=e^{-b[\frac{W^{(\gamma)\prime}(d)}{W^{(\gamma)}(d)}-\Phi(\gamma)]}
\end{align*}
    
	c) The law of the post-supremum process is a \(h\)-transform of the spectrally negative L\'{e}vy process killed at $T \wedge \tau_{b}^{+}$with
	$$
	h(x)=1-e^{-\Phi(\gamma)(b-x)}, \quad x \leq b.
	$$ 
	Equivalently, the post- $H_{S}$ process is equal in distribution to $b+\Gamma$, where the law of $\Gamma$ is the $h$-transform of the law of a spectrally negative L\'{e}vy process conditioned to stay negative and killed at $T \wedge \tau_{0}^{+}$.

    Moreover, the distribution of the infimum of the post-supremum process, $I_{H_{S}, T}$, when $S_{T}=b$ is given by
	\begin{equation}
		\mathbb{P}_{b}\left(I_{H_{S}, T} > a \mid S_{T}=b\right)=\frac{Z^{(\gamma)\prime}(b-a)-\left(Z^{(\gamma)}(b-a)-1\right) \frac{W^{(\gamma)\prime}(b-a)}{W^{(\gamma)}(b-a)}}{\Phi(\gamma)}
	\end{equation}
	for $b\geq a$ 
also 
\begin{align*}
\mathbb{P}_{b}\left\{M_{H_S, T}^{-}>d \mid S_T=b\right\}=1+\frac{\left(Z^{(\gamma)}(d)-1\right) W^{(\gamma)^{\prime}}(d)-\gamma\left(W^{(\gamma)}(d)\right)^2}{\Phi(\gamma) W^{(\gamma)}(d)}
\end{align*}
 and for all $d\geq0.$ For $u\geq 0,$
 \begin{align*}
\mathbb{P}\left\{M_{H_{S},T}^{+}\geq u \mid S_{T}=b \right\} =1-\frac{\gamma W^{(\gamma)}(u)}{\Phi(\gamma)Z^{(\gamma)}(u)}
\end{align*} 
Moreover for $a\leq 0 \leq b,$
\[
\begin{aligned}
\mathbb{P}\left(I_T>a\mid S_T=b\right)
&=
e^{\Phi(\gamma)b}
\frac{W^{(\gamma)}(-a)}
     {W^{(\gamma)}(b-a)}
\frac{
Z^{(\gamma)\prime}(b-a)
-
\left(Z^{(\gamma)}(b-a)-1\right)
\dfrac{W^{(\gamma)\prime}(b-a)}
      {W^{(\gamma)}(b-a)}
}{
\Phi(\gamma)
}.
\end{aligned}
\]
\end{theorem}
\begin{proof}
The path decompositions, the characterizations of the pre-supremum and the post-supremum processes, their independence and maximum drawdown distributions 
are obtained in Lemma 1 of \cite{VardarAcar2021}. Here, we prove the distributions of the infimum before
and after the time of the supremum as well as the maximum drawup distributions of each segment. Conditional on $S_T=b$, the transition semigroup of the
post-$H_S$ process is
$$
P_{t}(x, \mathrm{~d} y)=\frac{h(y)}{h(x)} \mathbb{P}_{x}\left\{X_{t} \in \mathrm{d} y, t<T\right\}
$$
where $h(x)=1-e^{-\Phi(\gamma)(b-x)}$ with $x, y<b$, and the entrance law is obtained as $x \rightarrow b-$.
Then, it follows by applying L'Hôpital’s Rule that
\begin{align}\label{postsupinf}
&\lim_{x \rightarrow b}\mathbb{P}_x\left(I_{H_{S}, T}> a \mid S_T=b \right)=\lim_{x \rightarrow b}\mathbb{P}_x\left(T<\tau_{b}^+\wedge\tau_a^-\mid T<\tau_{b}^+
\right)
=\lim_{x \rightarrow b}\frac{\mathbb{P}_x\left(T<\tau_{b}^+\wedge\tau_a^-
\right)}{\mathbb{P}_x\left(T<\tau_{b}^+\right)}\nonumber\\
&=\lim_{x \rightarrow b}\frac{\mathbb{P}_x\left\{T<\tau_a^{-}, \tau_a^{-}<\tau_b^{+}\right\}+\mathbb{P}_x\left\{T<\tau_b^{+}, \tau_b^{+}<\tau_a^{-}\right\}}{\mathbb{P}_x\left(T<\tau_{b}^+\right)}\nonumber\\
&= \lim_{x \rightarrow b}\frac{1-Z^{(\gamma)}(x-a)+\left(Z^{(\gamma)}(b-a)-1\right) \frac{W^{(\gamma)}(x-a)}{W^{(\gamma)}(b-a)}}{1-e^{-\Phi(\gamma)(b-x)}}\nonumber\\
&= \frac{Z^{(\gamma)\prime}(b-a)-\left(Z^{(\gamma)}(b-a)-1\right) \frac{W^{(\gamma)\prime}(b-a)}{W^{(\gamma)}(b-a)}}{\Phi(\gamma)}
\end{align} also conditional on $S_T=b$, the pre-$H_S$ process is an \(h\)-transform of the law of the spectrally negative L\'{e}vy process killed at $T \wedge \tau_{b}^{+}$, with
	$$
	h(x)=e^{-\Phi(\gamma)(b-x)}, \quad x \leq b.
	$$ and for $a \leq 0 \leq b$ \begin{align}
	\mathbb{P}_0\left(I_{H_{S}} > a \mid S_{T}=b\right)&=\mathbb{P}_0^{\Phi (\gamma)}\left(I_{\tau_{b}^{+}} > a \right)=\frac{W^{\Phi (\gamma)}(-a)}{W^{\Phi (\gamma)}(-a+b)}\nonumber\\
    &=\frac{e^{-\Phi (\gamma)(-a)}W^{ (\gamma)}(-a)}{e^{-\Phi (\gamma)(b-a)}W^{(\gamma)}(-a+b)}=\frac{e^{\Phi (\gamma)b}W^{ (\gamma)}(-a)}{W^{(\gamma)}(-a+b)}\nonumber
	\end{align}
Both probabilities converge to one as $a\to-\infty$. Moreover,
\[
\begin{aligned}
\mathbb{P}\left(I_T>a\mid S_T=b\right)
&=
\mathbb{P}\left(
I_{H_S}>a\mid S_T=b
\right)
\mathbb{P}\left(
I_{H_S,T}>a\mid S_T=b
\right)\\
&=e^{\Phi(\gamma)b}
\frac{W^{(\gamma)}(-a)}
     {W^{(\gamma)}(b-a)}
\frac{
Z^{(\gamma)\prime}(b-a)
-
\left(Z^{(\gamma)}(b-a)-1\right)
\dfrac{W^{(\gamma)\prime}(b-a)}
      {W^{(\gamma)}(b-a)}
}{
\Phi(\gamma)
}.
\end{aligned}
\]

Furthermore, for $u\ge b,$ the distribution of the maximum increase of the pre-supremum process is; 
\begin{align*}
\mathbb{P}\left\{M_{0, H_{S}}^{+}< u \mid S_{T}=b \right\}
&=\mathbb{P}_{0}^{h}\left\{ \tau_{b}^{+}<\tau_{b-u}^{-},\tau_{b}^{+}<T  \right\}\\
&=\frac{h(b)}{h(0)} \mathbb{P}_0\left(\tau_b^{+}<\tau_{b-u}^{-}<, \tau_b^{+}<T\right)\\
&=e^{b\Phi(\gamma)}\frac{W^{(\gamma)}(u-b)}{W^{(\gamma)}(u)}.
\end{align*} note that $0\leq\mathbb{P}\left\{M_{0, H_{S}}^{+}< u \mid S_{T}=b \right\} \leq 1$ and $\lim_{u \to
\infty}\mathbb{P}\left\{M_{0, H_{S}}^{+}< u \mid S_{T}=b \right\}=1.$ The distribution of the maximum increase of the post-supremum processes is found for $0<b-x<u,$
\begin{align*}
\mathbb{P}\left\{M_{H_{S},T}^{+}< u \mid S_{T}=b \right\}
&=\lim_{x\to b}\mathbb{P}_{x}( M_T^+< u|T<\tau_b^+)\\
&=\lim_{x\to b}\frac{\mathbb{P}_{x}( M_T^+< u,T<\tau_b^+)}{\mathbb{P}_{x}(T<\tau_b^+)}\\
&=\lim_{x\to b}\frac{Z^{(\gamma)}(u)-Z^{(\gamma)}(u-b+x)}
{Z^{(\gamma)}(u)[1-e^{-\Phi(\gamma)(b-x)}]}.
\end{align*} by Prop.2 \cite{Pistorious2004}.
Now applying L'H$\hat{o}$pital and using $Z^{(\gamma)\prime}(u)=\gamma W^{(\gamma)}(u)$ gives
\begin{align*}
\mathbb{P}\left\{M_{H_{S},T}^{+}< u \mid S_{T}=b \right\}=\lim_{x\to b}\mathbb{P}_{x}( M_T^+< u|T<\tau_b^+)=\frac{\gamma W^{(\gamma)}(u)}{\Phi(\gamma)Z^{(\gamma)}(u)}
\end{align*}
and note that $0\leq\mathbb{P}\left\{M_{H_{S},T}^{+}< u \mid S_{T}=b \right\} \leq 1$ and $\lim_{u \to
\infty}\mathbb{P}\left\{M_{H_{S},T}^{+}< u \mid S_{T}=b \right\}=1,$ by the fact that $\lim_{u \to
\infty}\frac{ Z^{(\gamma)}(u)}{ W^{(\gamma)}(u)}=\frac{\gamma}{\Phi(\gamma)},$ as given in Lemma 3.3 of \cite{rivero2013}.
Furthermore,
\begin{eqnarray}
\mathbb{P}\left\{M_{T}^{+}< u \mid S_{T}=b \right\}&=&\mathbb{P}\left\{M_{0, H_{S}}^{+}< u \mid S_{T}=b \right\}\mathbb{P}\left\{M_{H_{S},T}^{+}< u \mid S_{T}=b \right\}\nonumber\\
&=&e^{b\Phi(\gamma)}\frac{W^{(\gamma)}(u-b)}{W^{(\gamma)}(u)}\frac{\gamma W^{(\gamma)}(u)}{\Phi(\gamma)Z^{(\gamma)}(u)},\quad u\geq b
\end{eqnarray} and using density of the supremum and integrating over values of $b$ recovers the well known result $\mathbb{P}\left\{M_{T}^{+}< u \right\}=1-\frac{1}{Z^{(\gamma)}(u)}$ given in Prop.2 \cite{Pistorious2004}.
\end{proof}

\section{Path decomposition at the infimum}\label{sec:path-decomposition-at-infimum}

In this section, we next continue with presenting the decomposition of the process on the value of the infimum which identifies the characterizations of the pre-infumum and the post-infimum processes and finds the maximum drawdown and supremum laws in each of these components. These results were proved in Proposition 1 of \cite{VardarAcar2021} and Theorem 3.1 of \cite{VardarAcar2026}. Additional to recalling these results, we prove the maximum drawup distributions of each segment for the first time. This decomposition will later be used for finding the law of the additional conditioning $\mathbb{P}_0( H_I<H_S,I_T=a,S_T=b)$ imposed on the process which is one of the bulding blocks for finding the joint distribution of the maximum drawdown and the maximum drawup.  Note also that, these conditioning on the extrema values and conditioning on the order of them requires separate treatments as the spectrally negative Levy process has no positive jumps which violates the symmetry property.
\begin{theorem}\label{Prop_pre_post_inf_dist}
Conditionally on $I_T=a$, the pre-$H_I$ and post-$H_I$ processes are independent.

\begin{enumerate}[label=\textup{(\roman*)}]
\item The law of the pre-$H_I$ process is the Doob $h$-transform of the spectrally negative L\'{e}vy process killed at $\tau_a\wedge T$, with
\[
h(x)=\frac{\gamma}{\Phi(\gamma)}W^{(\gamma)\prime}(x-a)-\gamma W^{(\gamma)}(x-a),\qquad x\geq a.
\]
Moreover,
\begin{align*}
\mathbb{P}_0\left\{S_{H_I}<b\mid I_T=a\right\}
&=1-\frac{\frac{\gamma}{\Phi(\gamma)}W^{(\gamma)\prime}(b-a)-\gamma W^{(\gamma)}(b-a)}
{\frac{\gamma}{\Phi(\gamma)}W^{(\gamma)\prime}(-a)-\gamma W^{(\gamma)}(-a)}
\frac{W^{(\gamma)}(-a)}{W^{(\gamma)}(b-a)},
\qquad b\geq0,
\end{align*}
and
\begin{align*}
\mathbb{P}\left\{M_{0,H_I}^{-}<d\mid I_T=a\right\}
&=1-\frac{\frac{\gamma}{\Phi(\gamma)}W^{(\gamma)\prime}(d)-\gamma W^{(\gamma)}(d)}
{\frac{\gamma}{\Phi(\gamma)}W^{(\gamma)\prime}(-a)-\gamma W^{(\gamma)}(-a)}
\frac{W^{(\gamma)}(-a)}{W^{(\gamma)}(d)},
\qquad d\geq0.
\end{align*}
For $u>0$,
\begin{align*}
\mathbb{P}\left\{M_{0,H_I}^{+}\leq u \mid I_T=a\right\}
=\frac{\frac{\gamma W^{(\gamma)}(u)}{\Phi(\gamma)}\mathcal{L}_{\theta \rightarrow -a}^{-1}\left\{\frac{1}{Z_\gamma(u, \theta)}\right\}}{\frac{\gamma}{\Phi(\gamma)} W^{(\gamma)^{\prime}}(-a)-\gamma W^{(\gamma)}(-a)}
\end{align*}
where $\mathcal{L}^{-1}$ is the inverse Laplace transform and  $Z_\gamma(u, \theta)$ denotes the generalized/two-parameter scale function.
\item The law of the post-$H_I$ process is the Doob $h$-transform of the spectrally negative L\'{e}vy process killed at $T\wedge\tau_a^-$, with
\[
h(x)=1-Z^{(\gamma)}(x-a)+\frac{\gamma}{\Phi(\gamma)}W^{(\gamma)}(x-a).
\]
Moreover,
\[
\mathbb{P}\left\{S_{H_I,T}<b\mid I_T=a\right\}
=\frac{\Phi(\gamma)\bigl(Z^{(\gamma)}(b-a)-1\bigr)}{\gamma W^{(\gamma)}(b-a)},
\qquad b>a,
\]
\begin{align*}
\mathbb{P}\left\{M_{H_I,T}^{-}>d\mid I_T=a\right\}
&=1-\Phi(\gamma)\frac{W^{(\gamma)}(d)}{W^{(\gamma)\prime}(d)},
\qquad d\geq0,
\end{align*}
and, for $u>0$,
\begin{align*}
\mathbb{P}\left\{M_{H_I,T}^{+}< u\mid I_T=a\right\}
=\frac{\Phi(\gamma)\bigl(Z^{(\gamma)}(u)-1\bigr)}{\gamma W^{(\gamma)}(u)}.
\end{align*}
\end{enumerate}
\end{theorem}

\begin{proof}
Because $X$ has unbounded variation, zero is regular for both half-lines.  Conditional on $I_T=a$, Millar's decomposition implies that the minimum is attained at a unique time and $H_I=\tau_a$ almost surely; see \cite{Millar1977}.  The Markov property at $H_I$ then gives the conditional independence of the pre-$H_I$ and post-$H_I$ processes.

For the pre-infimum component, the conditional semigroup is obtained by differentiating the law of the terminal infimum with respect to the lower boundary.  For $x>a$, the corresponding excessive function is
\[
\frac{\mathrm d}{\mathrm da}\mathbb{P}_x\{I_T\leq a\}
=\frac{\gamma}{\Phi(\gamma)}W^{(\gamma)\prime}(x-a)-\gamma W^{(\gamma)}(x-a),
\]
which gives the stated Doob transform.  Applying the two-sided exit identity to this transform yields the distribution of $S_{H_I}$.  Since the pre-infimum path ends at $a$, the maximum-drawdown identity follows by applying the same exit at the level corresponding to a drawdown of size $d$.  This is the pre-infimum decomposition proved in \cite{VardarAcar2026}.
The post-infimum transform and the stated post-infimum extrema laws are those of \cite{VardarAcar2021}.  As novelty in this paper, it remains to record the two maximum-drawup derivations. 
And proving the maximum drawup law of the pre-infimum process uses the reflected process at the infimum and its regulator as follows, 
\begin{align*}
\mathbb{P}\left\{M_{x,H_{I}}^{+}\leq u \mid I_{T}=a \right\}
&=1-\frac{1}{h(x-a)} \int_{[0,x-a)} h(x-a+u-\ell) \eta_u^{(\gamma)}(d \ell),
\qquad a\leq x,\ u>0,
\end{align*}
with
\[
\eta_u^{(\gamma)}(d \ell):=\mathbb{E}_x\left[e^{-\gamma \hat{\sigma}_u} ; L_{\hat{\sigma}_u} \in d \ell\right]
\]
being the law of regulator $L$ where
\[
\hat{\sigma}_u:=\inf \left\{t \geq 0: Y_t \geq u\right\}
\]
is the first drawup time with
\[
\hat{Y}_t:=X_t-I_t=X_t+L_t
\]
is the reflected process at the infimum and its regulator
\[
L_t:=-I_t
\]
and $h$ excessive function is the the Doob $h$-transform of the pre-infimum process.

Starting from x, 
$$
\int_0^{\infty} e^{-\theta \ell} \eta_u^{(\gamma)}(d \ell)=\mathbb{E}_x\left[e^{-\gamma \hat{\sigma}_u-\theta L_{\hat{\sigma}_u}}\right]=\mathbb{E}_0\left[e^{-\gamma \hat{\sigma}_u-\theta L_{\hat{\sigma}_u}}\right]=\frac{1}{Z_\gamma(u, \theta)} ,
$$ Theorem 2, \cite{Ivanovs2012}.
Here $Z_\gamma(u, \theta)$ denotes the generalized/two-parameter scale function where
$$
Z_\gamma(u, \theta)=e^{\theta u}\left[1+(\gamma-\psi(\theta)) \int_0^u e^{-\theta y} W^{(\gamma)}(y) d y\right]
$$
In particular,
$$
Z_\gamma(u, 0)=1+\gamma \int_0^u W^{(\gamma)}(y) d y=Z^{(\gamma)}(u)
$$
therefore, using the inverse Laplace transform $\mathcal{L}^{-1}$ we have,
$$\eta_u^{(\gamma)}(d \ell)=\mathcal{L}_{\theta \rightarrow x-a}^{-1}\left\{\frac{1}{Z_\gamma(u, \theta)}\right\}$$
and hence the desired probability is,
\begin{equation}\label{preinfmaxdup}
\mathbb{P}\left\{M_{x,H_{I}}^{+}< u \mid I_{T}=a \right\}=\frac{\frac{\gamma W^{(\gamma)}(u)}{\Phi(\gamma)}\mathcal{L}_{\theta \rightarrow x-a}^{-1}\left\{\frac{1}{Z_\gamma(u, \theta)}\right\}}{\frac{\gamma}{\Phi(\gamma)} W^{(\gamma)^{\prime}}(x-a)-\gamma W^{(\gamma)}(x-a)}.
\end{equation}
On the other hand, obtaining the law of the maximum drawup of the post-infimum process implements the two sided exit probabilities as follows:
\begin{align*}
\mathbb{P}\left\{M_{H_{I},T}^{+}\geq u \mid I_{T}=a \right\}
&=\lim_{x\to a}\mathbb{P}_x^h\left(\tau_{a+u}^{+}<T,\tau_{a+u}^{+}<\tau_{a}^{-}\right)\\
&=\lim_{x\to a}\frac{h(a+u)}{h(x)}
\mathbb{P}_x\left(\tau_{a+u}^{+}<T,\tau_{a+u}^{+}<\tau_{a}^{-}\right)\\
&=\lim_{x\to a}
\frac{1-Z^{(\gamma)}(u)+\frac{\gamma}{\Phi(\gamma)} W^{(\gamma)}(u)}
{1-Z^{(\gamma)}(x-a)+\frac{\gamma}{\Phi(\gamma)} W^{(\gamma)}(x-a)}
\left[\frac{W^{(\gamma)}(x-a)}{W^{(\gamma)}(u)}\right]\\
&=1-
\frac{\Phi(\gamma)\bigl(Z^{(\gamma)}(u)-1\bigr)}
{\gamma W^{(\gamma)}(u)}.
\end{align*}
also
\[
0\leq\mathbb{P}\left\{M_{H_{I},T}^{+}\geq u \mid I_{T}=a \right\} \leq 1
\]
and
\[
\lim_{u\to 0}\mathbb{P}\left\{M_{H_{I},T}^{+}\geq u \mid I_{T}=a \right\}=1.
\]
Note that
\begin{align*}
\mathbb{P}_0\left\{M_{T}^{+}< u \right\}
&= \int_{-\infty}^0 \mathbb{P}\left(M_{0, H_I}^{+}<u \mid I_T=a\right)
\mathbb{P}\left(M_{H_I, T}^{+}<u \mid I_T=a\right) f_{I_T}(a) da\\
&=\int_{-\infty}^{0}
\frac{\frac{\gamma W^{(\gamma)}(u)}{\Phi(\gamma)}\mathcal{L}_{\theta \rightarrow -a}^{-1}\left\{\frac{1}{Z_\gamma(u, \theta)}\right\}}
{\frac{\gamma}{\Phi(\gamma)} W^{(\gamma)^{\prime}}(-a)-\gamma W^{(\gamma)}(-a)}\\
&\qquad\times
\frac{\Phi(\gamma)\bigl(Z^{(\gamma)}(u)-1\bigr)}{\gamma W^{(\gamma)}(u)}
\left[\frac{\gamma}{\Phi(\gamma)} W^{(\gamma)^{\prime}}(-a)-\gamma W^{(\gamma)}(-a)\right]da\\
&=\frac{Z^{(\gamma)}(u)-1}{Z^{(\gamma)}(u)}.
\end{align*}
which recovers the well-known formula of the maximum drawup of the entire process, Prop.2 \cite{Pistorious2004}. \end{proof}

\section{Path decomposition under $H_S<H_I$}\label{sec:HS-before-HI}

In our way to finding the joint distribution of the maximum drawdown and maximum drawup up to exponentially distributed time $T$ taken independent of the underlying path, we decompose the process in to three parts under the condition of the values of the infimum and the supremum and the ordering that the supremum is attained before the infimum, that is the event ${H_S<H_I, I_T=a, S_T=b},$ for $a \leq 0 \leq b.$  The path is decomposed into the pre-supremum, the intermediate, and the post-infimum components, and the characterizations of each one of these are provided in the following Theorem \ref{Prop_H_S<H_I}. Note that the spectrally negative Levy process has no positive jumps; therefore, the reversed order of the extrema decomposition can not be obtained from symmetry; each excessive function must be identified carefully and separately.
\begin{theorem}\label{Prop_H_S<H_I}
Fix $a<0<b$. Under $\mathbb{P}_{a,b}^{>}$, the pre-$H_S$ process, the intermediate process, and the post-$H_I$ process are independent and their laws are as follows;
\begin{enumerate}[label=\textup{(\roman*)}]
\item The pre-supremum process is the Doob transform of the spectrally negative L\'{e}vy process killed when it reaches $b$ before the exponential time $T$, with excessive function
\[
h^*(x)=\frac{\gamma}{\Phi(\gamma)}W^{(\gamma)}(x-a)
\frac{\left(W^{(\gamma)\prime}(b-a)\right)^2-W^{(\gamma)}(b-a)W^{(\gamma)\prime\prime}(b-a)}
{\left(W^{(\gamma)}(b-a)\right)^2},
\qquad a<x<b.
\]
\item The intermediate process is the Doob transform of the post-supremum process killed at $\tau_a$, with excessive function
\[
\begin{aligned}
h^{**}(x)
&=\frac{Z^{(\gamma)}(b-a)-1}{\left(W^{(\gamma)}(b-a)\right)^2}\\
&\quad\times\left[
W^{(\gamma)\prime}(x-a)W^{(\gamma)}(b-a)
-
W^{(\gamma)}(x-a)W^{(\gamma)\prime}(b-a)
\right],
\qquad a<x<b.
\end{aligned}
\]
\item The post-infimum process is the Doob transform of the spectrally negative L\'{e}vy process issued from $a$ and killed at $T\wedge\tau_b^+\wedge\tau_a^-$, with excessive function
\[
h^{***}(x)=1-Z^{(\gamma)}(x-a)
+\left(Z^{(\gamma)}(b-a)-1\right)
\frac{W^{(\gamma)}(x-a)}{W^{(\gamma)}(b-a)},
\qquad a<x<b.
\]
\end{enumerate}
\end{theorem}
\begin{proof}
i. For the pre-supremum component, we start with the pre-infimum law in Theorem~\ref{Prop_pre_post_inf_dist} and the further condition that the path reaches the prescribed supremum level $b$ before the infimum level $a$. That is,
For $  a<x,y<b$, we have
\begin{align*}  
P_t(x,\mathrm{d}y) &=  \mathbb{P}_x \{X_t\in \mathrm{d}y, t<H_S\, |\, H_S<H_I, S_T =b, I_T=a\} \\
&= \frac{ \mathbb{P}_x \{X_t\in \mathrm{d}y, t<H_S, H_S<H_I, S_T = b \, | \,I_T=a\} }{\mathbb{P}_x \{ H_S<H_I, S_T = b \, | \,I_T=a\}     }  \\
&= \frac{ \frac{\mathrm{d}}{\mathrm{d}b}\mathbb{P}_x^{h_1} \{ S_{\tau_a} \geq b \, | X_t\in \mathrm{d}y, t<\tau_b^+
\} \mathbb{P}_x^{h_1} \{X_t\in \mathrm{d}y, t<\tau_b^+\}}{\frac{\mathrm{d}}{\mathrm{d}b} \mathbb{P}_x^{h_1} \{  S_{\tau_a} \geq b \} }\\
&= \frac{ \frac{\mathrm{d}}{\mathrm{d}b}\mathbb{P}_y^{h_1} \{ S_{\tau_a} 
\geq b   \}} {\frac{\mathrm{d}}{\mathrm{d}b}\mathbb{P}_x^{h_1} \{ S_{\tau_a} \geq b  \}     } \mathbb{P}_x^{h_1} \{X_t\in \mathrm{d}y, t<\tau_b^+\}\\
&=\frac{\frac{\mathrm{d}}{\mathrm{d}b} \mathbb{P}_y^{h_1} \{S_{\tau_a} \geq b\} }{\frac{\mathrm{d}}{\mathrm{d}b}\mathbb{P}_x^{h_1} \{ S_{\tau_a} \geq b \}     } \frac{\frac{\gamma}{\Phi(\gamma)} W^{(\gamma)\prime}(y-a) -\gamma W^{(\gamma)}(y-a) }{\frac{\gamma}{\Phi(\gamma)} W^{(\gamma)\prime}(x-a) -\gamma W^{(\gamma)}(x-a) } \mathbb{P}_{x}\left\{X_{t} \in \mathrm{d} y, t<\tau_{b^+}\right\}.
\end{align*}
By \cite{VardarAcar2026},
$$
\mathbb{P}_x^{h_1} \{ S_{\tau_a} \geq b \}=\frac{ \frac{\gamma}{\Phi(\gamma)} W^{(\gamma)\prime}(b-a) -\gamma W^{(\gamma)}(b-a) }{\frac{\gamma}{\Phi(\gamma)} W^{(\gamma)\prime}(x-a) -\gamma W^{(\gamma)}(x-a)}\left[\frac{W^{(\gamma)}(x-a)}{W^{(\gamma)}(b-a)}\right]
$$ 

Consequently,
\begin{align*}
h^*(x)&=-\frac{\mathrm{d}}{\mathrm{d}b}\mathbb{P}_x^{h_1} \{ S_{\tau_a} \geq b \}h_1(x)=-\frac{\mathrm{d}}{\mathrm{d}b}\mathbb{P}_x^{h_1} \{ S_{\tau_a} \geq b \}[\frac{\gamma}{\Phi(\gamma)} W^{(\gamma)\prime}(x-a) -\gamma W^{(\gamma)}(x-a)]\\
&=\frac{\gamma}{\Phi(\gamma)} W^{(\gamma)}(x-a) 
\frac{ \left(W^{(\gamma)\prime}(b-a)\right)^2 - W^{(\gamma)}(b-a)W^{(\gamma)\prime\prime}(b-a) }{ \left(W^{(\gamma)}(b-a)\right)^2} 
\end{align*}
Here,
$h_1(x)=\frac{\gamma}{\Phi(\gamma)}W^{(\gamma)\prime}(x-a)
-\gamma W^{(\gamma)}(x-a)$ is the excessive function for the pre-infimum
process.

In summary, differentiating the corresponding conditional supremum law with respect to $b$ and multiplying by the pre-infimum excessive function gives
\[
\frac{\gamma}{\Phi(\gamma)}W^{(\gamma)}(x-a)
\frac{\left(W^{(\gamma)\prime}(b-a)\right)^2-W^{(\gamma)}(b-a)W^{(\gamma)\prime\prime}(b-a)}
{\left(W^{(\gamma)}(b-a)\right)^2},
\]
which proves the first assertion.

ii. When $H_S<H_I,S_{T}=b,I_{T}=a$, the transition semigroup of the intermediate process in view of Theorem \ref{lemma_pre_post_sup_dist}, for $a<x,y<b$ is  
$$
\begin{aligned}
&P_t(x, \mathrm{~d} y)=\frac{\mathbb{P}_x\left\{X_t \in \mathrm{~d} y, t<\tau_a, I_{H_S,T} \leq a, S_T=b\right\}}{\mathbb{P}_x\left\{I_{H_S,T} \leq a, S_T=b\right\}} \\
& =\frac{\mathbb{P}_x\left\{I_{H_S,T} \leq a, S_T=b\mid X_t=y, t<\tau_a\right\} \mathbb{P}_x^h\left\{X_t \in \mathrm{~d} y, t<\tau_a\right\}}{\mathbb{P}_x\left\{I_{H_S,T} \leq a, S_T=b\right\}} \\
&=\frac{\mathbb{P}_y\left\{I_{\tau_b^+,T} \leq a\mid S_T=b\right\} (1-e^{-\Phi(\gamma)(b-y)})\mathbb{P}_x\left\{X_t \in \mathrm{~d} y, t<\tau_a\right\}}{(1-e^{-\Phi(\gamma)(b-x)})\mathbb{P}_x\left\{I_{\tau_b^+,T} \leq a\mid S_T=b\right\}} 
\end{aligned}
$$
where $h(x)=1-e^{-\Phi(\gamma)(b-x)}, x \leq b$ is the \(h\)-transform of post-supremum process. 

By Theorem \ref{lemma_pre_post_sup_dist} given above, we have
\begin{align}
\mathbb{P}_{x}\left(I_{H_{S}, T} \leq a \mid S_{T}=b\right)=1-\frac{1-Z^{(\gamma)}(x-a)+\left(Z^{(\gamma)}(b-a)-1\right) \frac{W^{(\gamma)}(x-a)}{W^{(\gamma)}(b-a)}}{1-e^{-\Phi(\gamma)(b-x)}}
\end{align}

\begin{align}
\frac{\mathrm{d}}{\mathrm{d}a}\mathbb{P}_{x}\left(I_{H_{S}, T} \leq a \mid S_{T}=b\right)
&=\frac{Z^{(\gamma)}(b-a)-1}
{\left(1-e^{-\Phi(\gamma)(b-x)}\right)\left(W^{(\gamma)}(b-a)\right)^2}
\nonumber\\
&\quad\times\left[
W^{(\gamma)\prime}(x-a)W^{(\gamma)}(b-a)
-
W^{(\gamma)}(x-a)W^{(\gamma)\prime}(b-a)
\right].\nonumber
\end{align}
As a result, the distribution of the intermediate process is Doob $h$-transform of the law of the original Lévy process killed at $\tau_a$ with
\begin{align}
h^{**}(x)&:=\frac{\mathrm{d}}{\mathrm{d}a}\mathbb{P}_{x}\left(I_{H_{S}, T} \leq a \mid S_{T}=b\right)\left(1-e^{-\Phi(\gamma)(b-x)}\right)\nonumber\\
&= \frac{
\left(Z^{(\gamma)}(b-a)-1\right)
\left[
W^{(\gamma)\prime}(x-a)W^{(\gamma)}(b-a)
-
W^{(\gamma)}(x-a)W^{(\gamma)\prime}(b-a)
\right]
}{
\left(W^{(\gamma)}(b-a)\right)^2
}.
\end{align} 
iii. Finally, the post-infimum component is the original process conditioned to remain inside $(a,b)$ until the exponential killing time.  By the two-sided exit identities,
$$\begin{aligned}
		& h^{***}(x)=\mathbb{P}_{x}\left\{T<\tau_{b}^{+} \wedge \tau_{a}^{-}\right\} =\mathbb{P}_{x}\left\{\tau_{b}^{+}>T \right\}-\mathbb{P}_{x}\left\{\tau_{a}^{-}<T, \tau_{b}^{+}>T\right\}\\&=\mathbb{P}_{x}\left\{\tau_{b}^{+}>T \right\}-\mathbb{P}_{x}\left\{\tau_{a}^{-}<T\right\}+\mathbb{P}_{x}\left\{\tau_{a}^{-}<T, \tau_{b}^{+}<T\right\}\\
		& =\mathbb{P}_{x}\left\{\tau_{b}^{+}>T \right\}-\mathbb{P}_{x}\left\{\tau_{a}^{-}<T\right\}+\mathbb{P}_{x}\left\{\tau_{a}^{-}<T, \tau_{b}^{+}<\tau_{a}^{-}\right\}+\mathbb{P}_{x}\left\{\tau_{b}^{+}<T, \tau_{a}^{-}<\tau_{b}^{+}\right\} \\
        &=1-\mathbb{P}_{x}\left\{\tau_{a}^{-}<T, \tau_{a}^{-}<\tau_{b}^{+}\right\}-\mathbb{P}_{x}\left\{\tau_{b}^{+}<T, \tau_{b}^{+}<\tau_{a}^{-}\right\}\\
        &=1- [{Z^{(\gamma)}(x-a) - Z^{(\gamma)}(b-a) \frac{W^{(\gamma)}(x-a)}{W^{(\gamma)}(b-a)}}]-[\frac{W^{(\gamma)}(x-a)}{W^{(\gamma)}(b-a)}]\\
        &=1-Z^{(\gamma)}(x-a)+\left(Z^{(\gamma)}(b-a)-1\right) \frac{W^{(\gamma)}(x-a)}{W^{(\gamma)}(b-a)},
		\end{aligned}
	$$
In other words for, $a<x,y<b$, we have
\begin{align*}  
P_t(x,\mathrm{d}y) &=  \mathbb{P}_x \{X_t\in \mathrm{d}y, t<T<\tau_{b}^{+} \wedge \tau_{a}^{-}\, |\, H_S<H_I, S_T =b, I_T=a\} \\
&= \frac{ \mathbb{P}_y \{ T<\tau_{b}^{+} \wedge \tau_{a}^{-}\, | X_t\in \mathrm{d}y, t<T<\tau_{b}^{+} \wedge \tau_{a}^{-}\,
\} \mathbb{P}_x \{X_t\in \mathrm{d}y, t<T<\tau_{b}^{+} \wedge \tau_{a}^{-}\,\}}{\mathbb{P}_x \{ T<\tau_{b}^{+} \wedge \tau_{a}^{-} \} }\\
&= \frac{ \mathbb{P}_y \{  T<\tau_{b}^{+} \wedge \tau_{a}^{-}  \}} {\mathbb{P}_x \{ T<\tau_{b}^{+} \wedge \tau_{a}^{-} \}     } \mathbb{P}_x \{X_t\in \mathrm{d}y, t<T<\tau_{b}^{+} \wedge \tau_{a}^{-}\,\}\\
&=\frac{1-Z^{(\gamma)}(y-a) + (Z^{(\gamma)}(b-a)-1) \frac{W^{(\gamma)}(y-a)}{W^{(\gamma)}(b-a)}
        }{1-Z^{(\gamma)}(x-a) + (Z^{(\gamma)}(b-a)-1) \frac{W^{(\gamma)}(x-a)}{W^{(\gamma)}(b-a)}}
        \mathbb{P}_{x}\left\{X_{t} \in \mathrm{d} y, t<T<\tau_{b}^{+} \wedge \tau_{a}^{-}\,\right\}\\
        &=\frac{h^{* * *}(y)}{h^{* * *}(x)} \mathbb{P}_x\left\{X_t \in d y, t<T \wedge \tau_a^{-} \wedge \tau_b^{+}\right\}.
\end{align*} 
The conditional independence follows from the preceding splitting-at-the-$I_T/S_T$ decomposition together with the Markov property of the corresponding conditioned components.
\end{proof}

\begin{remark} For standard Brownian motion reflecting the path using symmetry does not change its distribution but reverses the ordering of the infimum and the supremum.  Substituting \eqref{eq:brownian-scale-functions} into the three excessive functions in Theorem \ref{Prop_H_S<H_I}, for the pre-supremum component,
\[
2\gamma\,
\frac{\sinh\!\left(\sqrt{2\gamma}(x-a)\right)}
{\sinh^2\!\left(\sqrt{2\gamma}(b-a)\right)},
\qquad a<x<b,
\]
whose logarithmic derivative is
\[
\sqrt{2\gamma}\,\coth\!\left(\sqrt{2\gamma}(x-a)\right).
\]
For the intermediate component the logarithmic derivative reduces to
\[
-\sqrt{2\gamma}\,\coth\!\left(\sqrt{2\gamma}(b-x)\right),
\]
and the post-infimum excessive function becomes
\[
\frac{
\sinh\!\left(\sqrt{2\gamma}(b-a)\right)
-\sinh\!\left(\sqrt{2\gamma}(x-a)\right)
-\sinh\!\left(\sqrt{2\gamma}(b-x)\right)
}{
\sinh\!\left(\sqrt{2\gamma}(b-a)\right)
}.
\]
Therefore, the generators of the h-transforms of each component coincide with the reflected counterparts of the Brownian decomposition found in \cite{Salminen2007}.
\end{remark}
\section{Path decomposition under $H_I<H_S$}\label{sec:HI-before-HS}

We next condition on both values of the supremum and the infimum and on the ordering in which the infimum is attained before the supremum, that is fix $a<0<b$ and work under $\mathbb{P}_{a,b}^{<}$. Here, the path splits into the pre-infimum, the intermediate component from $a$ to $b$, and the post-supremum component.  
\begin{theorem}\label{prop:HI-before-HS}
Under $\mathbb{P}_{a,b}^{<}$, the pre-$H_I$ process, the intermediate process
\[
\{X_{H_I+u}:0\leq u\leq H_S-H_I\},
\]
and the post-$H_S$ process are independent. Their laws are as follows.
\begin{enumerate}[label=\textup{(\roman*)}]
\item The pre-infimum process is the Doob transform of the process killed on leaving $(a,b)$ through the lower boundary, with
\begin{equation}\label{eq:h1-HIHS}
h_1(x)=
\frac{W^{(\gamma)\prime}(x-a)W^{(\gamma)}(b-a)-W^{(\gamma)\prime}(b-a)W^{(\gamma)}(x-a)}
{\bigl(W^{(\gamma)}(b-a)\bigr)^2},
\qquad a<x<b.
\end{equation}

\item The intermediate process has the law of $a+\Gamma$, where $\Gamma$ is the Doob transform of a spectrally negative L\'{e}vy process killed at $\tau_{b-a}^+\wedge\tau_0^-$, with
\begin{equation}\label{eq:h2-HIHS}
h_2(x)=e^{-\Phi(\gamma)(b-a-x)}\frac{W^{(\gamma)}(x)}{W^{(\gamma)}(b-a)},
\qquad 0<x<b-a.
\end{equation}

\item The post-$H_S$ process is the Doob transform of the process killed on leaving $(a,b)$ before $T$, with
\begin{equation}\label{eq:h3-HIHS}
h_3(x)=1-Z^{(\gamma)}(x-a)+\bigl(Z^{(\gamma)}(b-a)-1\bigr)
\frac{W^{(\gamma)}(x-a)}{W^{(\gamma)}(b-a)},
\qquad a<x<b.
\end{equation}
\end{enumerate}
\end{theorem}

\begin{proof} For the proof and the characterization of the pre-$H_I$ process we refer the reader to \cite{VardarAcar2026}. And the characterizations of the intermediate and post-supremum components are proved in \cite{VardarAcar2021}.
\end{proof}
\section{Maximum drawups under $H_S<H_I$ and Maximum drawdowns under $H_I<H_S$ }

Having identified the three independent components of the spectrally negative Levy process conditioned under both events $I_T=a, S_T=b, H_S<H_I$ and ${I_T=a, S_T=b, H_I<H_S}$ in previous Sections~\ref{sec:HS-before-HI} and \ref{sec:HI-before-HS}, we now compute the maximum drawup and maximum drawdown laws on each segment, respectively. We begin our proof by finding the joint law of the maximum gain and the running supremum down to $\tau_a^-,$ $(M_{\tau_a^-}^{+},S_{\tau_a^-}),$ which will later be used for finding the maximum gain law of the intermediate process under the case where $H_S<H_I$.
\begin{lemma}
Consider a spectrally negative L\'{e}vy process $X$ with $\gamma$-scale functions $W^{(\gamma)}$ and $Z^{(\gamma)}$. Let $a<0$ and $\gamma\geq0$. Then, for $u>0$,
\begin{eqnarray}\label{goingdownmaxdup}
\mathbb E_x\left[
e^{-\gamma\tau_a^-};
M_{\tau_a^-}^{+}\leq u
\right]=1-Z^{(\gamma)}(u)
\mathcal L_{\theta\to x-a}^{-1}
\left\{
\frac1{\theta Z_\gamma(u,\theta)}
\right\}\end{eqnarray} where \[
Z_\gamma(u,\theta)
=
e^{\theta u}
\left[
1+(\gamma-\psi(\theta))
\int_0^u e^{-\theta y}W^{(\gamma)}(y)\,dy
\right]
\]
\end{lemma}
\begin{proof}
 Consider the
process reflected at its running infimum,
\[
\widehat Y_t:=X_t-I_t,
\]
with associated regulator
$
L_t:=x-I_t.
$
Thus $\widehat Y_0=L_0=0$ under $\mathbb P_x$.

For $u>0$, define the first passage time of
the reflected process above level $u$ by
\[
\widehat\sigma_u
:=
\inf\{t\geq0:\widehat Y_t\geq u\}.
\] On
$\{\tau_a^-<\infty\}$,
\[
\{M_{\tau_a^-}^{+}<u\}
=
\{\tau_a^-<\widehat\sigma_u\}.
\]
Then it follows that
\begin{equation}\label{eq:drawup-downcross-event}
\begin{aligned}
&\mathbb E_x\left[
e^{-\gamma\tau_a^-};
M_{\tau_a^-}^{+}<u
\right]\\
&\qquad =
\mathbb E_x\left[
e^{-\gamma\tau_a^-};
\tau_a^-<\widehat\sigma_u
\right].
\end{aligned}
\end{equation}

Observe that, at the first down-crossing of $a$,
the process is at its running infimum, so that
\[
X_{\tau_a^-}=I_{\tau_a^-},
\qquad
\widehat Y_{\tau_a^-}=0,
\]
on $\{\tau_a^-<\infty\}$.
This remains valid when the down-crossing occurs
by a negative jump.

Applying the strong Markov property at
$\tau_a^-$, 
\begin{equation}\label{eq:drawup-strong-markov}
\begin{aligned}
&\mathbb E_x\left[
e^{-\gamma\widehat\sigma_u};
\tau_a^-<\widehat\sigma_u
\right]\\
&\quad =
\mathbb E_x\left[
e^{-\gamma\tau_a^-};
\tau_a^-<\widehat\sigma_u
\right]
\mathbb E_0\left[
e^{-\gamma\widehat\sigma_u}
\right].
\end{aligned}
\end{equation}

To evaluate the last expectation, we use
Theorem~2 of \cite{Ivanovs2012}, specialized
to a spectrally negative L\'evy process
reflected at its infimum. This result gives
the joint Laplace transform of the first
passage time and the regulator:
\begin{equation}\label{eq:ivanovs-transform}
\mathbb E_0\left[
e^{-\gamma\widehat\sigma_u
-\theta L_{\widehat\sigma_u}}
\right]
=
\frac{1}{Z_\gamma(u,\theta)},
\qquad \theta\geq0,
\end{equation}
where \[
Z_\gamma(u,\theta)
=
e^{\theta u}
\left[
1+(\gamma-\psi(\theta))
\int_0^u e^{-\theta y}W^{(\gamma)}(y)\,dy
\right]
\]
In particular, setting $\theta=0$ and using
$Z_\gamma(u,0)=Z^{(\gamma)}(u)$ yields
\begin{equation}\label{eq:reflected-passage-lt}
\mathbb E_0\left[
e^{-\gamma\widehat\sigma_u}
\right]
=
\frac{1}{Z^{(\gamma)}(u)}.
\end{equation}

Combining
\eqref{eq:drawup-downcross-event},
\eqref{eq:drawup-strong-markov}, and
\eqref{eq:reflected-passage-lt}, we obtain
\begin{equation}\label{eq:drawup-regulator-step}
\begin{aligned}
\mathbb E_x\left[
e^{-\gamma\tau_a^-};
M_{\tau_a^-}^{+}<u
\right] =
Z^{(\gamma)}(u)
\mathbb E_x\left[
e^{-\gamma\widehat\sigma_u};
\tau_a^-<\widehat\sigma_u
\right].
\end{aligned}
\end{equation}

We next express the expectation on the
right-hand side in terms of the regulator.
 Since $X$ has unbounded
variation, $0$ is regular for
$(-\infty,0)$. Therefore, the event that
the process crosses below $a$ before
$\widehat\sigma_u$ satisfies
\[
\{\tau_a^-<\widehat\sigma_u\}
=
\{L_{\widehat\sigma_u}>x-a\},
\].

Define the regulator measure by
\begin{equation}\label{eq:regulator-measure}
\eta_u^{(\gamma)}(B)
:=
\mathbb E_0\left[
e^{-\gamma\widehat\sigma_u};
L_{\widehat\sigma_u}\in B
\right],
\end{equation}
for Borel sets $B\subseteq[0,\infty)$.

Then
\begin{equation}\label{eq:regulator-tail}
\begin{aligned}
&\mathbb E_x\left[
e^{-\gamma\widehat\sigma_u};
\tau_a^-<\widehat\sigma_u
\right]\\
&\qquad =
\eta_u^{(\gamma)}
\bigl((x-a,\infty)\bigr).
\end{aligned}
\end{equation}

Taking $\theta=0$ in
\eqref{eq:ivanovs-transform}, we have
\[
\eta_u^{(\gamma)}([0,\infty))
=
\frac{1}{Z^{(\gamma)}(u)}.
\]
Hence
\begin{equation}\label{eq:regulator-tail-complement}
\begin{aligned}
&\eta_u^{(\gamma)}
\bigl((x-a,\infty)\bigr)\\
&\quad =
\frac{1}{Z^{(\gamma)}(u)}
-
\eta_u^{(\gamma)}
\bigl([0,x-a]\bigr).
\end{aligned}
\end{equation}

By Tonelli's theorem, for $\theta>0$,
\begin{align*}
\int_0^\infty
e^{-\theta y}
\eta_u^{(\gamma)}([0,y])\,dy
=
\frac1\theta
\int_{[0,\infty)}
e^{-\theta\ell}
\eta_u^{(\gamma)}(d\ell)=
\frac{1}{\theta Z_\gamma(u,\theta)}.
\end{align*}
Consequently, 
\begin{equation}\label{eq:regulator-inversion}
\eta_u^{(\gamma)}([0,y])
=
\mathcal L^{-1}_{\theta\to y}
\left\{
\frac{1}{\theta Z_\gamma(u,\theta)}
\right\}.
\end{equation}

Finally, substituting
\eqref{eq:regulator-tail},
\eqref{eq:regulator-tail-complement},
and \eqref{eq:regulator-inversion} into
\eqref{eq:drawup-regulator-step} gives
\begin{align*}
&\mathbb E_x\left[
e^{-\gamma\tau_a^-};
M_{\tau_a^-}^{+}<u
\right]\\
&\quad =
Z^{(\gamma)}(u)
\left[
\frac{1}{Z^{(\gamma)}(u)}
-
\eta_u^{(\gamma)}([0,x-a])
\right]\\
&\quad =
1-Z^{(\gamma)}(u)
\mathcal L^{-1}_{\theta\to x-a}
\left\{
\frac{1}{\theta Z_\gamma(u,\theta)}
\right\},
\end{align*} which proves \ref{goingdownmaxdup}.
\end{proof}

\begin{proposition}\label{thm:q-drawup-supremum-tauminus}
Consider a spectrally negative L\'{e}vy process $X$ with $\gamma$-scale functions $W^{(\gamma)}$ and $Z^{(\gamma)}$. Let $a<0$, $b>0$, $x\in[0,b]$, and $\gamma\geq0$. Then, for $u>0$,
\[
\begin{aligned}
&\mathbb E_x\left[e^{-\gamma\tau_a^-};\,M_{\tau_a^-}^{+}\leq u,\ S_{\tau_a^-}\leq b\right]\\
&\qquad=
\begin{cases}
\displaystyle
1-Z^{(\gamma)}(u)
\mathcal L_{\theta\to x-a}^{-1}
\left\{
\frac1{\theta Z_\gamma(u,\theta)}
\right\},
& 0<u\leq b-x,\\[1.3em]
\displaystyle
\begin{aligned}
&Z^{(\gamma)}(x-b+u)
-
Z^{(\gamma)}(u)
\mathcal L^{-1}_{\theta\to(b-a-u)^-}
\left\{
\frac{
Z^{(\gamma)}(x-b+u,\theta)
}{
\theta Z^{(\gamma)}(u,\theta)
}
\right\},
\end{aligned}
& b-x<u\leq b-a,\\[1.3em]
\displaystyle
\begin{aligned}
&Z^{(q)}(x-a)-Z^{(q)}(b-a)\frac{W^{(q)}(x-a)}{W^{(q)}(b-a)},
\end{aligned}& b-a\leq u.
\end{cases}
\end{aligned}
\]

\end{proposition}

\begin{proof} By the result given in previous Lemma \ref{goingdownmaxdup}, we have
\begin{equation}
    \mathbb E_x\left[
e^{-\gamma\tau_a^-};
M_{\tau_a^-}^{+}\leq u
\right]=1-Z^{(\gamma)}(u)
\mathcal L_{\theta\to x-a}^{-1}
\left\{
\frac1{\theta Z_\gamma(u,\theta)}
\right\}
\label{eq:mdu-before-downcrossing}
\end{equation} 
We now impose the additional restriction $S_{\tau_a^-}\leq b$.  Let $Y_t:=X_t-I_t,~t\geq0,$ be the process reflected at its running infimum. Then $ M_{\tau_{a}^-}^{+}=\sup_{0\leq t\leq\tau_a}Y_t.$

\noindent\emph{Case 1: For $0<u\leq b-x$,}
on $\{M_{\tau_a^-}^{+}\leq u\}$, for every
$t\leq\tau_a^-$,
\[
X_t=I_t+Y_t\leq I_t+u\leq x+u\leq b.
\]
Hence
\[
\{M_{\tau_a^-}^{+}\leq u\}
\subseteq
\{S_{\tau_a^-}\leq b\},
\] that is, the drawup restriction itself implies $S_{\tau_a^-}\leq b$, and therefore Equation \eqref{eq:mdu-before-downcrossing} gives the first branch.
\smallskip

\noindent\emph{Case 2: For $u\geq b-a$,}
On $\{S_{\tau_a^-}\leq b\}$, the path remains in $(a,b]$ before
$\tau_a^-$, and therefore
\[
X_t-I_t\leq b-a\leq u,
\qquad 0\leq t\leq\tau_a^-.
\]
Thus
\[
\{S_{\tau_a^-}\leq b\}
\subseteq
\{M_{\tau_a^-}^{+}\leq u\}.
\]
 that is the condition $S_{\tau_a^-}\leq b$ implies the event $M_{\tau_a^-}^{+}\leq u$.  Hence the event reduces to $\{S_{\tau_a^-}\leq b\}=\{\tau_a^-<\tau_b^+\},$ and this two-sided exit identity gives
\[
Z^{(q)}(x-a)-Z^{(q)}(b-a)\frac{W^{(q)}(x-a)}{W^{(q)}(b-a)}.
\]
\smallskip

\noindent\emph{Case 3: Finally for $b-x<u<b-a$,} we have $a<b-u<x<b$. 

Before \(\tau_{b-u}^-\wedge\tau_b^+\), the running infimum remains at least \(b-u\), while the process remains below \(b\). Consequently,
\[
X_t-\underline X_t<u,
\qquad
t<\tau_{b-u}^-\wedge\tau_b^+.
\]

The process must cross below \(b-u\) before reaching \(a\) and before crossing above \(b\).

Moreover, at \(\tau_{b-u}^-\), the process attains a new running infimum, possibly through a downward jump. Since no drawup of size \(u\) has occurred before this stopping time, and the process restarts from a new running infimum, the drawup restriction for the complete path is equivalent to the corresponding restriction for the restarted process.

Applying the tower property and the strong Markov property at \(\tau_{b-u}^-\), we obtain
\begin{eqnarray}
&&\mathbb E_x\Bigg[
\mathbf1_{\{\tau_{b-u}^-<\tau_b^+\}}\mathbb E_x\Big[
e^{-\gamma\tau_a^-}
\mathbf1_{\{\tau_a^-<\widehat\alpha_u,\,
\tau_a^-<\tau_b^+\}}
\;\Big|\;
\mathcal F_{\tau_{b-u}^-}
\Big]
\Bigg]\nonumber\\
&=&\mathbb E_x\Bigg[
e^{-\gamma\tau_{b-u}^-}
\mathbf1_{\{\tau_{b-u}^-<\tau_b^+\}}
\left.
\mathbb E_z\left[
e^{-\gamma\tau_a^-};
\tau_a^-<\widehat\alpha_u,\,
\tau_a^-<\tau_b^+
\right]
\right|_{z=X_{\tau_{b-u}^-}}
\Bigg].\end{eqnarray}

Since \(X_{\tau_{b-u}^-}\le b-u\), the condition \(\tau_a^-<\tau_b^+\) can be omitted from the inner expectation. Indeed, a process started from \(z\le b-u\) cannot reach \(b\) without first producing a drawup of at least \(u\).

Thus we have,
\begin{eqnarray}\label{withouttauataub}
\mathbb E_x\Bigg[
e^{-\gamma\tau_{b-u}^-}
\mathbf1_{\{\tau_{b-u}^-<\tau_b^+\}}\left.
\mathbb E_z\left[
e^{-\gamma\tau_a^-};
\tau_a^-<\widehat\alpha_u
\right]
\right|_{z=X_{\tau_{b-u}^-}}
\Bigg]\end{eqnarray}

By Lemma \ref{goingdownmaxdup}, for \(z>a\),
\begin{eqnarray}\mathbb E_z\left[
e^{-\gamma\tau_a^-};
\tau_a^-<\widehat\alpha_u
\right]=
1-Z^{(\gamma)}(u)
\mathcal L^{-1}_{\theta\to z-a}
\left\{
\frac{1}{\theta Z^{(\gamma)}(u,\theta)}
\right\}.\end{eqnarray}

For \(z\le a\), the first down-crossing time is zero, and the inner expectation in \ref{withouttauataub} equals one.

 Therefore, Equation \ref{withouttauataub} equals,
\begin{eqnarray}
\mathbb E_x\left[
e^{-\gamma\tau_{b-u}^-};
\tau_{b-u}^-<\tau_b^+
\right]-
Z^{(\gamma)}(u)
\mathbb E_x\Bigg[
e^{-\gamma\tau_{b-u}^-}
\mathbf1_{\{\tau_{b-u}^-<\tau_b^+\}}
\mathbf1_{\{X_{\tau_{b-u}^-}>a\}}
\left.
\mathcal L^{-1}_{\theta\to z-a}
\left\{
\frac{1}{\theta Z^{(\gamma)}(u,\theta)}
\right\}
\right|_{z=X_{\tau_{b-u}^-}}
\Bigg]\nonumber\end{eqnarray}
The first expectation is evaluated by the standard two-sided exit identity. Consider the second expectation. At \(\tau_{b-u}^-\), the downward overshoot satisfies
\[
b-u-X_{\tau_{b-u}^-}\ge0.
\]

The restriction \(X_{\tau_{b-u}^-}>a\) is equivalent to
\[
b-u-X_{\tau_{b-u}^-}<b-a-u.
\]

Furthermore, the discounted overshoot distribution is characterized by the two-sided exit identity \cite[Theorem~6.5]{avram2020},  
\begin{eqnarray}
   && \mathbb E_x\left[
e^{-\gamma\tau_{b-u}^-
+\theta(X_{\tau_{b-u}^-}-(b-u))};
\tau_{b-u}^-<\tau_b^+
\right]\\
&=&
Z^{(\gamma)}(x-b+u,\theta)-
\frac{W^{(\gamma)}(x-b+u)}
{W^{(\gamma)}(u)}
Z^{(\gamma)}(u,\theta).\nonumber\end{eqnarray}

Consequently, the second expectation including the factor \(Z^{(\gamma)}(u)\), can be written as
\begin{eqnarray}\label{secondexpectationwithfactor}
&&Z^{(\gamma)}(u)
\mathbb E_x\Bigg[
e^{-\gamma\tau_{b-u}^-}
\mathbf1_{\{\tau_{b-u}^-<\tau_b^+\}}
\mathbf1_{\{X_{\tau_{b-u}^-}>a\}}
\left.
\mathcal L^{-1}_{\theta\to z-a}
\left\{
\frac{1}{\theta Z^{(\gamma)}(u,\theta)}
\right\}
\right|_{z=X_{\tau_{b-u}^-}}
\Bigg]\\&=&
Z^{(\gamma)}(u)
\int_{[0,b-a-u)}
\mathcal L^{-1}_{\lambda\to b-a-u-s}
\left\{
\frac{1}{\lambda Z^{(\gamma)}(u,\lambda)}
\right\}
\mathcal L^{-1}_{\theta\to s}
\Bigg\{
Z^{(\gamma)}(x-b+u,\theta)
-
\frac{W^{(\gamma)}(x-b+u)}
{W^{(\gamma)}(u)}
Z^{(\gamma)}(u,\theta)
\Bigg\}(ds).\nonumber\end{eqnarray}

The integral is the convolution of the discounted overshoot measure and the cumulative regulator distribution appearing in Lemma 7.1.

By the convolution property of Laplace transforms, \ref{secondexpectationwithfactor} equals
\begin{eqnarray}\label{afterconv}
&&Z^{(\gamma)}(u)
\mathcal L^{-1}_{\theta\to(b-a-u)^-}
\Bigg\{
\frac{1}{\theta Z^{(\gamma)}(u,\theta)}
\left[
Z^{(\gamma)}(x-b+u,\theta)
-
\frac{W^{(\gamma)}(x-b+u)}
{W^{(\gamma)}(u)}
Z^{(\gamma)}(u,\theta)
\right]
\Bigg\}\nonumber\\
&=&Z^{(\gamma)}(u)
\mathcal L^{-1}_{\theta\to(b-a-u)^-}
\Bigg\{
\frac{Z^{(\gamma)}(x-b+u,\theta)}
{\theta Z^{(\gamma)}(u,\theta)}
-
\frac{W^{(\gamma)}(x-b+u)}
{\theta W^{(\gamma)}(u)}
\Bigg\}.
\end{eqnarray}
Since \(b-a-u>0\), we have
\[
\mathcal L^{-1}_{\theta\to(b-a-u)^-}
\left\{\frac{1}{\theta}\right\}=1.
\]

Therefore, \ref{afterconv} reduces to
\begin{eqnarray}
Z^{(\gamma)}(u)
\mathcal L^{-1}_{\theta\to(b-a-u)^-}
\left\{
\frac{Z^{(\gamma)}(x-b+u,\theta)}
{\theta Z^{(\gamma)}(u,\theta)}
\right\}-
Z^{(\gamma)}(u)
\frac{W^{(\gamma)}(x-b+u)}
{W^{(\gamma)}(u)}.
\end{eqnarray}

Finally, substituting the two sided exit identity and \ref{afterconv} into \ref{withouttauataub} yields
\begin{eqnarray}&&E_x\left[
e^{-\gamma\tau_a^-};
\tau_a^-<\widehat\alpha_u,\,
\tau_a^-<\tau_b^+
\right]\nonumber\\
&=&
Z^{(\gamma)}(x-b+u)
-
Z^{(\gamma)}(u)
\frac{W^{(\gamma)}(x-b+u)}
{W^{(\gamma)}(u)}\\
&-&
Z^{(\gamma)}(u)
\mathcal L^{-1}_{\theta\to(b-a-u)^-}
\left\{
\frac{Z^{(\gamma)}(x-b+u,\theta)}
{\theta Z^{(\gamma)}(u,\theta)}
\right\}
+
Z^{(\gamma)}(u)
\frac{W^{(\gamma)}(x-b+u)}
{W^{(\gamma)}(u)}\nonumber\\
&=&Z^{(\gamma)}(x-b+u)
-
Z^{(\gamma)}(u)
\mathcal L^{-1}_{\theta\to(b-a-u)^-}
\left\{
\frac{
Z^{(\gamma)}(x-b+u,\theta)
}{
\theta Z^{(\gamma)}(u,\theta)
}
\right\}.\nonumber
\end{eqnarray}
This completes the proof of Case 3. 
\end{proof}
\begin{proposition}\label{prop:joint-drawup-extrema}
Consider a spectraly negative  L\'evy process of unbounded variation, $X=(X_t)_{t\geq0}$ with $\gamma$-scale functions $W^{(\gamma)}$ and $Z^{(\gamma)}.$  For $a<0<b$, suppose that
$
0<u<b-a,~b-u<x<b.
$
Then
\begin{eqnarray}
\label{eq:joint-drawup-extrema}
\mathbb P_x\left(
M_T^+<u,\ I_T>a,\ S_T<b
\right)
=
1-Z^{(\gamma)}(x-b+u)
+
\bigl[Z^{(\gamma)}(u)-1\bigr]
\mathcal L^{-1}_{\theta\to(b-a-u)}
\left\{
\frac{Z_\gamma(x-b+u,\theta)}
{\theta Z_\gamma(u,\theta)}
\right\},\nonumber
\end{eqnarray}
where
\[
Z_\gamma(v,\theta)
=
e^{\theta v}
\left[
1+(\gamma-\psi(\theta))
\int_0^v e^{-\theta z}
W^{(\gamma)}(z)\,dz
\right].
\]
\end{proposition}

\begin{proof}
Let
\[
\widehat\alpha_u
=
\inf\{t\geq0:X_t-I_t\geq u\}
\]
denote the first drawup time of size $u,$ we have
\begin{eqnarray}
\mathbb P_x\left(
M_T^+<u,\ I_T>a,\ S_T<b
\right)
=
\mathbb P_x\left(
T<\widehat\alpha_u\wedge
\tau_a^-\wedge\tau_b^+
\right).
\label{eq:joint-step1}
\end{eqnarray}

Since $b-u<x<b$, the process cannot attain
a drawup of size $u$ before leaving the
interval $(b-u,b)$. Furthermore, because
$a<b-u$, a downward passage below $a$
necessarily involves a passage below $b-u$.

We therefore apply the strong Markov property
at $\tau_{b-u}^-$ and
we obtain
\begin{eqnarray}\label{eq:joint-step2}
&&\mathbb P_x\left(
T<\widehat\alpha_u\wedge
\tau_a^-\wedge\tau_b^+
\right)\\
&=&
\mathbb P_x\left(
T<\tau_{b-u}^-\wedge\tau_b^+
\right)
+
\mathbb E_x\Bigg[
e^{-\gamma\tau_{b-u}^-}
\mathbf1_{\{
\tau_{b-u}^-<\tau_b^+,\,
X_{\tau_{b-u}^-}>a
\}}
\mathbb P_{X_{\tau_{b-u}^-}}
\left(
T<\widehat\alpha_u\wedge\tau_a^-
\right)
\Bigg].\nonumber
\end{eqnarray}
By the reflected-process identity in
Lemma~\ref{goingdownmaxdup}, for $z>a$,
\begin{eqnarray}
&\mathbb E_z\left[
e^{-\gamma\tau_a^-};
\tau_a^-<\widehat\alpha_u
\right]=
1-Z^{(\gamma)}(u)
\eta_u^{(\gamma)}([0,z-a)),
\label{eq:joint-step3}
\end{eqnarray}
where 
\begin{equation}
\label{eq:joint-eta-transform}
\int_{[0,\infty)}
e^{-\theta\ell}
\eta_u^{(\gamma)}(d\ell)
=
\frac{1}{Z_\gamma(u,\theta)}.
\end{equation}

Moreover,
\begin{eqnarray}
\mathbb E_z\left[
e^{-\gamma\widehat\alpha_u};
\widehat\alpha_u<\tau_a^-
\right]
=
\eta_u^{(\gamma)}([0,z-a)).
\label{eq:joint-step4}
\end{eqnarray}

Decomposing according to the first of
$\widehat\alpha_u$, $\tau_a^-$ and $T$,
and using \eqref{eq:joint-step3} and
\eqref{eq:joint-step4}, we obtain
\begin{eqnarray}
\mathbb P_z\left(
T<\widehat\alpha_u\wedge\tau_a^-
\right)
&=&
1-
\mathbb E_z\left[
e^{-\gamma\tau_a^-};
\tau_a^-<\widehat\alpha_u
\right]
-
\mathbb E_z\left[
e^{-\gamma\widehat\alpha_u};
\widehat\alpha_u<\tau_a^-
\right]
\nonumber\\
&=&
\bigl[Z^{(\gamma)}(u)-1\bigr]
\eta_u^{(\gamma)}([0,z-a)).
\label{eq:joint-step5}
\end{eqnarray}

Substituting \eqref{eq:joint-step5} into
\eqref{eq:joint-step2} gives
\begin{eqnarray}\label{eq:joint-step6}
&&\mathbb P_x\left(
T<\widehat\alpha_u\wedge
\tau_a^-\wedge\tau_b^+
\right)\\
&=&
\mathbb P_x\left(
T<\tau_{b-u}^-\wedge\tau_b^+
\right)+
\bigl[Z^{(\gamma)}(u)-1\bigr]
\mathbb E_x\Big[
e^{-\gamma\tau_{b-u}^-};
\tau_{b-u}^-<\tau_b^+,
X_{\tau_{b-u}^-}>a;\
\eta_u^{(\gamma)}
([0,X_{\tau_{b-u}^-}-a))
\Big].\nonumber
\end{eqnarray}

By the standard two-sided exit identities,
\begin{eqnarray}\label{eq:joint-step9}
\mathbb P_x\left(
T<\tau_{b-u}^-\wedge\tau_b^+
\right)
=
1-Z^{(\gamma)}(x-b+u)
+
\bigl[Z^{(\gamma)}(u)-1\bigr]
\frac{W^{(\gamma)}(x-b+u)}
{W^{(\gamma)}(u)}.
\end{eqnarray}

Therefore, it remains to calculate the expectation
in the second term of
\eqref{eq:joint-step6}.

For $\theta>\Phi(\gamma)$, the discounted
overshoot transform at $\tau_{b-u}^-$, \cite[Theorem~6.5]{avram2020} is, 
\begin{eqnarray}\label{eq:joint-step10}
\mathbb E_x\Big[
e^{-\gamma\tau_{b-u}^-+
\theta(X_{\tau_{b-u}^-}-(b-u))};
\tau_{b-u}^-<\tau_b^+
\Big]
=
Z_\gamma(x-b+u,\theta)
-
\frac{W^{(\gamma)}(x-b+u)}
{W^{(\gamma)}(u)}
Z_\gamma(u,\theta).\nonumber
\end{eqnarray} 
Applying the Laplace convolution 
to \eqref{eq:joint-step10} and
\eqref{eq:joint-eta-transform} yields
\begin{eqnarray}\label{eq:joint-step11}
&&\mathbb E_x\Big[
e^{-\gamma\tau_{b-u}^-};
\tau_{b-u}^-<\tau_b^+,
X_{\tau_{b-u}^-}>a;\
\eta_u^{(\gamma)}
([0,X_{\tau_{b-u}^-}-a))
\Big]
\nonumber\\
&=&
\mathcal L^{-1}_{\theta\to(b-a-u)}
\left\{
\frac{1}{\theta Z_\gamma(u,\theta)}
\left[
Z_\gamma(x-b+u,\theta)
-
\frac{W^{(\gamma)}(x-b+u)}
{W^{(\gamma)}(u)}
Z_\gamma(u,\theta)
\right]
\right\}
\nonumber\\
&=&
\mathcal L^{-1}_{\theta\to(b-a-u)}
\left\{
\frac{Z_\gamma(x-b+u,\theta)}
{\theta Z_\gamma(u,\theta)}
\right\}
-
\frac{W^{(\gamma)}(x-b+u)}
{W^{(\gamma)}(u)}.
\end{eqnarray}

Here the last equality follows from
\[
\mathcal L^{-1}_{\theta\to y}
\left\{\frac1\theta\right\}=1,
\qquad y>0,
\]
and $b-a-u>0$.

Finally, substituting
\eqref{eq:joint-step9} and
\eqref{eq:joint-step11} into
\eqref{eq:joint-step6}, we obtain

\begin{eqnarray}
\mathbb P_x\left(
T<\widehat\alpha_u\wedge
\tau_a^-\wedge\tau_b^+
\right)
&=&
1-Z^{(\gamma)}(x-b+u)\\
&+&
\bigl[Z^{(\gamma)}(u)-1\bigr]
\mathcal L^{-1}_{\theta\to(b-a-u)}
\left\{
\frac{Z_\gamma(x-b+u,\theta)}
{\theta Z_\gamma(u,\theta)}
\right\}.\nonumber
\end{eqnarray}
\end{proof}
\begin{theorem}\label{theorem:mdu-HSHI}
Fix $a<0<b$. Under $\mathbb{P}_{a,b}^{>}$, the following identities hold.
\begin{enumerate}[label=\textup{(\roman*)}]
\item For the pre-supremum process and $b<u<b-a$,
\begin{align}\label{eq:mdu-pre-supremum}
\mathbb{P}_{a,b}^{>}\left\{M_{0,H_S}^{+}\leq u\right\}
=\dfrac{W^{(\gamma)}(b-a)}{W^{(\gamma)}(-a)}
\dfrac{W^{(\gamma)}(u-b)}{W^{(\gamma)}(u)}.
\end{align}

\item For the intermediate process and $0<u\leq b-a$,
\begin{eqnarray}
\mathbb P_{a,b}^{>}
\left(M_{H_S,H_I}^{+}<u\right)
=
\frac{
[W^{(\gamma)}(u)]^2\mathcal L^{-1}_{\theta\to(b-a-u)}
\left\{
\frac{\psi(\theta)-\gamma}
{Z_\gamma(u,\theta)}
\right\}
}{
\displaystyle
\frac{[W^{(\gamma)\prime}(b-a)]^2}
{W^{(\gamma)}(b-a)}
-
W^{(\gamma)\prime\prime}(b-a)
}
.
\end{eqnarray}

\item For the post-infimum process and $0<u<b-a$,
\begin{align}\label{eq:mdu-post-infimum}
\mathbb{P}_{a,b}^{>}\left\{M_{H_I,T}^{+} \leq u\right\}
=\dfrac{W^{(\gamma)}(b-a)\left[Z^{(\gamma)}(u)-1\right]}
{W^{(\gamma)}(u)\left[Z^{(\gamma)}(b-a)-1\right]}.
\end{align}
\end{enumerate}
\end{theorem}

\begin{proof}
i. For the pre-supremum component, the Doob-transform factors in Theorem~\ref{Prop_H_S<H_I} cancel when the probability of reaching $b$ before $b-u$ is divided by the probability of reaching $b$ before $a$. That is, for $b<u<b-a$, we have  
\begin{align}
\mathbb{P}_{a,b}^{>}\left\{M_{0,H_S}^{+}<u\right\}
&=\mathbb{P}^{h^*}_0(\tau_b^+< \tau_{b-u}^-|\tau_b^+< \tau_{a}^-,\tau_b^+< T)\nonumber\\&=\frac{\frac{h^*(b)}{h^*(0)}}{\frac{h*(b)}{h^*(0)}}\frac{\mathbb{P}_0(\tau_b^+< \tau_{b-u}^-,\tau_b^+< T)}{\mathbb{P}_0(\tau_b^+< \tau_{a}^-,\tau_b^+< T)}\nonumber\\
&= \frac{W^{(\gamma)}(b-a)}{W^{(\gamma)}(-a)}\frac{W^{(\gamma)}(u-b)}{W^{(\gamma)}(u)}\nonumber
\end{align}

ii. For the intermediate component, the path is sure decrease from level $b$ to level $a$.  Applying Proposition~\ref{prop:joint-drawup-extrema} and dividing by the corresponding two-sided exit probabilityand letting the initial level tend to $b$ by applying L'H\^{o}pital's rule gives the corresponding result. Notice that,
\begin{eqnarray}
\mathbb{P}_{a,b}^{>}\left\{M_{H_S,H_I}^{+}<u\right\}
&=&
\lim_{x\uparrow b}
\frac{
\mathbb P_x\left(
M_{0,H_I}^{+}<u,\,
S_{0,H_I}<b
\mid I_T=a
\right)
}{
\mathbb P_x\left(
S_{0,H_I}<b
\mid I_T=a
\right)
}.
\nonumber
\end{eqnarray}
And by the independence of the pre-infimum and the post-infimum processes given in Theorem \ref{Prop_pre_post_inf_dist} we have,
\begin{eqnarray}
&-&\frac{\partial 
\mathbb P_x\left(
M_T^{+}<u,\,
I_T> a,\,
S_T<b
\right)
}
{\partial a}\\
&=&
\mathbb P_x\left(
M_{0,H_I}^{+}<u,\,
S_{0,H_I}<b
\mid I_T=a
\right)
\mathbb P_a\left(
M_{H_I,T}^{+}<u,\,
S_{H_I,T}<b
\mid I_T=a
\right)
\left[
\frac{\gamma}{\Phi(\gamma)}
W^{(\gamma)\prime}(x-a)
-\gamma W^{(\gamma)}(x-a)
\right].\nonumber
\end{eqnarray}
For \(0<u<b-a\), \[
\{M_{H_I,T}^{+}<u\}
\subseteq
\{S_{H_I,T}<b\},
\] therefore 
\begin{equation}\mathbb P_a\left(
M_{H_I,T}^{+}<u,\,
S_{H_I,T}<b
\mid I_T=a
\right)=
\mathbb P_a\left(
M_{H_I,T}^{+}<u
\mid I_T=a
\right)
=
\frac{\Phi(\gamma)[Z^{(\gamma)}(u)-1]}
{\gamma W^{(\gamma)}(u)}
\end{equation} by the result obtained in Theorem \ref{Prop_pre_post_inf_dist}(ii).

Thus,
\begin{eqnarray}\mathbb P_x\left(
M_{0,H_I}^{+}<u,\,
S_{0,H_I}<b
\mid I_T=a
\right)=\frac{-\frac{\partial 
\mathbb P_x\left(
M_T^{+}<u,\,
I_T> a,\,
S_T<b
\right)}{\partial a}}{\frac{\Phi(\gamma)[Z^{(\gamma)}(u)-1]}
{\gamma W^{(\gamma)}(u)}\left[
\frac{\gamma}{\Phi(\gamma)}
W^{(\gamma)\prime}(x-a)
-\gamma W^{(\gamma)}(x-a)
\right]}\nonumber\end{eqnarray}
and by the Propositon \ref{prop:joint-drawup-extrema}, we have
\begin{eqnarray}
\mathbb P_x\left(
M_T^{+}<u,\,
I_T> a,\,
S_T<b
\right)=1-Z^{(\gamma)}(x-b+u)
+
\bigl[Z^{(\gamma)}(u)-1\bigr]
\mathcal L^{-1}_{\theta\to(b-a-u)}
\left\{
\frac{Z_\gamma(x-b+u,\theta)}
{\theta Z_\gamma(u,\theta)}
\right\},\nonumber
\end{eqnarray}
where
\[
Z_\gamma(v,\theta)
=
e^{\theta v}
\left[
1+(\gamma-\psi(\theta))
\int_0^v e^{-\theta z}
W^{(\gamma)}(z)\,dz
\right]
\]
Thus, 
\begin{eqnarray}
\mathbb P_{a,b}^{>}
\left(M_{H_S,H_I}^{+}<u\right)
=
\lim_{x\uparrow b}
\frac{
\displaystyle
W^{(\gamma)}(u)
\mathcal L^{-1}_{\theta\to(b-a-u)}
\left\{
\frac{Z_\gamma(x-b+u,\theta)}
{Z_\gamma(u,\theta)}
\right\}
}{
\displaystyle
W^{(\gamma)\prime}(x-a)
-
\frac{W^{(\gamma)\prime}(b-a)}
{W^{(\gamma)}(b-a)}
W^{(\gamma)}(x-a)
}
\end{eqnarray}
and finally applying the L'Hôpital's rule yields
\begin{eqnarray}
\mathbb P_{a,b}^{>}
\left(M_{H_S,H_I}^{+}<u\right)
=
\frac{
[W^{(\gamma)}(u)]^2
}{
\displaystyle
\frac{[W^{(\gamma)\prime}(b-a)]^2}
{W^{(\gamma)}(b-a)}
-
W^{(\gamma)\prime\prime}(b-a)
}
\mathcal L^{-1}_{\theta\to(b-a-u)}
\left\{
\frac{\psi(\theta)-\gamma}
{Z_\gamma(u,\theta)}
\right\}.
\end{eqnarray}


iii. For the post-infimum process, start the killed process $x>a$ and use the excessive function in Theorem~\ref{Prop_H_S<H_I}. Then for $u<b-a$
\begin{align*}
\mathbb{P}_{a,b}^{>}\left(M_{H_I,T}^{+}<u\right) &= \lim\limits_{x \rightarrow a} \dfrac{\mathbb{P}_x(T < \tau_{u+a}^+,  T<\tau_b^+\wedge \tau_a^- )}{\mathbb{P}_x(T<\tau_b^+ \wedge \tau_a^- )}\\&=  \lim\limits_{x \rightarrow a} \dfrac{1-Z^{(\gamma)}(x-a)+[Z^{(\gamma)}(u+a-a)-1]\dfrac{W^{(\gamma)}(x-a)}{W^{(\gamma)}(u+a-a)}}{1-Z^{(\gamma)}(x-a)+[Z^{(\gamma)}(b-a)-1]\dfrac{W^{(\gamma)}(x-a)}{W^{(\gamma)}(b-a)}}\\ &=\lim\limits_{x \rightarrow a}\dfrac{Z^{(\gamma)\prime}(x-a)+[Z^{(\gamma)}(u)-1]\dfrac{W^{(\gamma)\prime}(x-a)}{W^{(\gamma)}(u)}}{Z^{(\gamma)\prime}(x-a)+[Z^{(\gamma)}(b-a)-1]\dfrac{W^{(\gamma)\prime}(x-a)}{W^{(\gamma)}(b-a)}}\\&=\dfrac{Z^{(\gamma)\prime}(0)+[Z^{(\gamma)}(u)-1]\dfrac{W^{(\gamma)\prime}(0)}{W^{(\gamma)}(b-a)}}{Z^{(\gamma)\prime}(0)+[Z^{(\gamma)}(b-a)-1]\dfrac{W^{(\gamma)\prime}(0)}{W^{(\gamma)}(b-a)}}=\dfrac{W^{(\gamma)}(b-a)[Z^{(\gamma)}(u)-1]}{W^{(\gamma)}(u)[Z^{(\gamma)}(b-a)-1]}
\end{align*}\end{proof}
\begin{remark}
For standard Brownian motion, substitution of \eqref{eq:brownian-scale-functions} into Theorem~\ref{theorem:mdu-HSHI} gives
\[
\mathbb{P}_{a,b}^{>}\left\{M_{0,H_S}^{+}\leq u\right\}
=
\frac{\sinh\!\left(\sqrt{2\gamma}(b-a)\right)}{-\sinh\!\left(\sqrt{2\gamma}a\right)}
\frac{\sinh\!\left(\sqrt{2\gamma}(u-b)\right)}{\sinh\!\left(\sqrt{2\gamma}u\right)},
\]
for the pre-supremum component,
\[\mathbb P_{a,b}^{>}
\left(M_{H_S,H_I}^{+}<u\right)=
\frac{
\sinh\bigl(\sqrt{2\gamma}(b-a)\bigr)
}{
\sinh\bigl(\sqrt{2\gamma}\,u\bigr)
}
\exp\left\{
-(b-a-u)\sqrt{2\gamma}
\coth\bigl(\sqrt{2\gamma}\,u\bigr)
\right\},
\]
for the intermediate component, and
\[
\mathbb{P}_{a,b}^{>}\left\{M_{H_I,T}^{+}<u\right\}
=
\frac{\cosh\!\left(\sqrt{2\gamma}u\right)-1}
{\cosh\!\left(\sqrt{2\gamma}(b-a)\right)-1}
\]
for the post-infimum component. Thus, the results found above in  Theorem~\ref{theorem:mdu-HSHI} verify the three laws of the symmetric counterparts of the maximum-decrease distributions of standard Brownian motion found in \cite[Thm.~4.1]{Salminen2007}.
\end{remark}
\begin{theorem}\label{theorem:mdd-HIHS}
Fix $a<0<b$. Under $\mathbb{P}_{a,b}^{<}$, the following identities hold.
\begin{enumerate}[label=\textup{(\roman*)}]
\item For the pre-infimum process,
{\small\begin{equation*}\label{eq:f1-piecewise}
\begin{aligned}
&\mathbb{P}_{a,b}^{<}\{M_{0,H_I}^-<d\}\\
&\qquad=\begin{cases}
0, & d\le -a,\\
\displaystyle
\begin{aligned}
1-&\left[\frac{W^{(\gamma)\prime}(d)W^{(\gamma)}(b-a)-W^{(\gamma)\prime}(b-a)W^{(\gamma)}(d)}
{W^{(\gamma)\prime}(-a)W^{(\gamma)}(b-a)-W^{(\gamma)\prime}(b-a)W^{(\gamma)}(-a)}\right]\\
&\quad\times\left(\frac{W^{(\gamma)}(-a)}{W^{(\gamma)}(d)}\right),
\end{aligned}
& -a<d<b-a,\\
1, & d\ge b-a.
\end{cases}
\end{aligned}
\end{equation*}}

\item For the intermediate process,
\begin{equation*}\label{eq:f2-piecewise}
\mathbb{P}_{a,b}^{<}\{M_{H_I,H_S}^-<d\}=\begin{cases}
0, & d\le0,\\[4pt]
\displaystyle
\frac{W^{(\gamma)}(b-a)}{W^{(\gamma)}(d)}
\exp\left\{-(b-a-d)\frac{W_+^{(\gamma)\prime}(d)}{W^{(\gamma)}(d)}\right\},
& 0<d<b-a,\\[12pt]
1, & d\ge b-a.
\end{cases}
\end{equation*}

\item For the post-supremum process,
\begin{equation*}\label{eq:f3-piecewise}
\mathbb{P}_{a,b}^{<}\{M_{H_S,T}^-<d\}=\begin{cases}
0, & d\le0,\\[4pt]
\displaystyle
\frac{\bigl(Z^{(\gamma)}(d)-1\bigr)\dfrac{W^{(\gamma)\prime}(d)}{W^{(\gamma)}(d)}-\gamma W^{(\gamma)}(d)}
{\bigl(Z^{(\gamma)}(b-a)-1\bigr)\dfrac{W^{(\gamma)\prime}(b-a)}{W^{(\gamma)}(b-a)}-\gamma W^{(\gamma)}(b-a)},
& 0<d<b-a,\\[14pt]
1, & d\ge b-a.
\end{cases}
\end{equation*}
\end{enumerate}
\end{theorem}

\begin{proof}
For the pre-infimum process, the event $\{M_{0,H_I}^-<d\}$ is impossible when $d\leq-a$, because the path starts from $0$ and ends at $a$.  If $-a<d<b-a$, Theorem~\ref{prop:HI-before-HS} and the two-sided exit identity give
\begin{align*}
\mathbb{P}_{a,b}^{<}\{M_{0,H_I}^-<d\}
&=1-
\frac{h_1(a+d)}{h_1(0)}
\frac{W^{(\gamma)}(-a)}{W^{(\gamma)}(d)}\\
&=1-\left[\frac{W^{(\gamma)\prime}(d)W^{(\gamma)}(b-a)-W^{(\gamma)\prime}(b-a)W^{(\gamma)}(d)}{W^{(\gamma)\prime}(-a)W^{(\gamma)}(b-a)-W^{(\gamma)\prime}(b-a)W^{(\gamma)}(-a)}\right]
\left(\frac{W^{(\gamma)}(-a)}{W^{(\gamma)}(d)}\right).
\end{align*}
The pre-$H_I$ component's maximum drawdown distribution is proved in \cite{VardarAcar2026}. For the intermediate component the reflected-process exit identity is applied to the transform in \eqref{eq:h2-HIHS} to find the maximum drawdown distribution. We refer the reader to the study \cite{VardarAcar2021} for the intermediate and the post-$H_S$ processes maximum drawdown distributions. 
\end{proof}
\section{Joint and marginal distributions of the maximum drawdown and the maximum drawup}\label{sec:joint-distributions-of-maximum-drawdowns-and-maximum-drawups}

The preceding sections provide the conditional laws of all path decompositions required for finding the unconditional joint law of the maximum drawdown and the maximum drawup of spectrally negative Levy process taken up to independent exponentially distributed time $T.$ We now continue our solution first by deriving the distribution of the conditions imposed on the process and densities of the extrema, followed by finding the joint laws marginal laws together with the ordering of these extrema, and finally by obtaining the joint law of maximum drawdown and maximum drawup free of any condition. 

Based on the conditional characterizations, the probability
\[
\mathbb{P}\{H_I < H_S,\, I_T=a,\, S_T=b\},
\]
which corresponds to the joint distribution of the infimum and supremum together with the event that the infimum occurs before the supremum, is derived in the following Corollary \ref{corollary_hinf<hsup,inf=a,sup=b}.
\begin{corollary}\label{corollary_hinf<hsup,inf=a,sup=b}
For \(a<0<b\),
\begin{align*}
p_{IS}^{<}(a,b)
&=
\frac{W^{(\gamma)\prime}(-a)\,W^{(\gamma)}(b-a)-
W^{(\gamma)}(-a)\,W^{(\gamma)\prime}(b-a)}
{\bigl(W^{(\gamma)}(b-a)\bigr)^2}\\
&\quad\times\left[Z^{(\gamma)\prime}(b-a)-\left(Z^{(\gamma)}(b-a)-1\right)
\frac{W^{(\gamma)\prime}(b-a)}{W^{(\gamma)}(b-a)}\right]
\end{align*}
\end{corollary}
\begin{proof}
By the independence and conditional laws in
Theorem~\ref{lemma_pre_post_sup_dist},
\begin{align*}
p_{IS}^{<}(a,b)
&=\frac{\mathrm{d}}{\mathrm{d}a}\mathbb{P}_0(I_{H_S}\leq a\mid S_T=b)
\mathbb{P}_b(I_{H_S,T}>a\mid S_T=b)\frac{\mathrm d}{\mathrm d b}\mathbb{P}_0(S_T\leq b) \\
&=\frac{\mathrm{d}}{\mathrm{d}a}\left(1-\frac{e^{\Phi(\gamma)b}W^{(\gamma)}(-a)}{W^{(\gamma)}(b-a)}\right)\left[Z^{(\gamma)\prime}(b-a)
-\frac{Z^{(\gamma)}(b-a)-1}{\Phi(\gamma)}
\frac{W^{(\gamma)\prime}(b-a)}{W^{(\gamma)}(b-a)}\right]
\Phi(\gamma)e^{-\Phi(\gamma)b}\\
&=
\frac{W^{(\gamma)\prime}(-a)\,W^{(\gamma)}(b-a)-
W^{(\gamma)}(-a)\,W^{(\gamma)\prime}(b-a)}
{\bigl(W^{(\gamma)}(b-a)\bigr)^2}\left[Z^{(\gamma)\prime}(b-a)-\left(Z^{(\gamma)}(b-a)-1\right)
\frac{W^{(\gamma)\prime}(b-a)}{W^{(\gamma)}(b-a)}\right]
\end{align*}

Note that, using the Brownian scale functions and derivatives in
\eqref{eq:brownian-scale-functions}, this result becomes
\[2\gamma\,
\frac{
\sinh\left(\sqrt{2\gamma}\,b\right)
\left(\cosh\left(\sqrt{2\gamma}(b-a)\right)-1\right)
}{
\sinh^3\left(\sqrt{2\gamma}(b-a)\right)
}
\]
which matches exactly with the result reported in \cite{Salminen2007}.
\end{proof}
Similarly, the probability
\[
\mathbb{P}\{H_S < H_I,\, I_T=a,\, S_T=b\},
\]
that corresponds to the joint distribution of the infimum and supremum with the supremum occurs before the infimum is obtained as follows in Corollary \ref{corollary_hsup<hinf,inf=a,sup=b}.
\begin{corollary}\label{corollary_hsup<hinf,inf=a,sup=b}
For \(a<0<b\),
\begin{align}\label{f(a,b)2}
&p_{IS}^{>}(a,b) \nonumber
= \frac{
W^{(\gamma)}(-a)\left(Z^{(\gamma)}(b-a)-1\right)
\left[
\left(W^{(\gamma)\prime}(b-a)\right)^2
-
W^{(\gamma)}(b-a)W^{(\gamma)\prime\prime}(b-a)
\right]
}{
\left(W^{(\gamma)}(b-a)\right)^3
}
\end{align}
\end{corollary}
\begin{proof}
By the independence of the pre-infimum and post-infimum processes in
Theorem~\ref{Prop_pre_post_inf_dist},
\begin{align*}
p_{IS}^{>}(a,b)
&=\frac{\mathrm{d}}{\mathrm{d}b}\mathbb{P}_{0}\left(S_{H_{I}}\leq b \mid I_{T}=a\right)
\mathbb{P}_{a}\left(S_{H_{I},T}<b\mid I_{T}=a\right)\frac{\mathrm{d}}{\mathrm{d}a}\mathbb{P}_0(I_T<a)\\&
=\frac{\mathrm{d}}{\mathrm{d}b}[1-\frac{ \frac{\gamma}{\Phi(\gamma)} W^{(\gamma)\prime}(b-a) -\gamma W^{(\gamma)}(b-a) }{\frac{\gamma}{\Phi(\gamma)} W^{(\gamma)\prime}(-a) -\gamma W^{(\gamma)}(-a)}\left[\frac{W^{(\gamma)}(-a)}{W^{(\gamma)}(b-a)}\right]]
(\frac{\Phi(\gamma)({Z^{(\gamma)}}(b-a)-1)}{\gamma W^{(\gamma)}(b-a)} )
\\&
\frac{\gamma}{\Phi(\gamma)} W^{(\gamma)\prime}(-a) -\gamma W^{(\gamma)}(-a)\\&
=\frac{
\frac{\gamma}{\Phi(\gamma)}W^{(\gamma)}(-a)
}{
\frac{\gamma}{\Phi(\gamma)}W^{(\gamma)\prime}(-a)
-\gamma W^{(\gamma)}(-a)
}
\,
\frac{
\left(W^{(\gamma)\prime}(b-a)\right)^2
-
W^{(\gamma)}(b-a)W^{(\gamma)\prime\prime}(b-a)
}{
\left(W^{(\gamma)}(b-a)\right)^2
}\\&
(\frac{\Phi(\gamma)({Z^{(\gamma)}}(b-a)-1)}{\gamma W^{(\gamma)}(b-a)} )
\frac{\gamma}{\Phi(\gamma)} W^{(\gamma)\prime}(-a) -\gamma W^{(\gamma)}(-a)\\&
=\frac{
W^{(\gamma)}(-a)\left(Z^{(\gamma)}(b-a)-1\right)
\left[
\left(W^{(\gamma)\prime}(b-a)\right)^2
-
W^{(\gamma)}(b-a)W^{(\gamma)\prime\prime}(b-a)
\right]
}{
\left(W^{(\gamma)}(b-a)\right)^3
}
\end{align*}
Note also that, using \eqref{eq:brownian-scale-functions}, this becomes
\[2\gamma\,
\frac{
\left(\cosh\left(\sqrt{2\gamma}(b-a)\right)-1\right)
\sinh\left(-\sqrt{2\gamma}a\right)
}{
\sinh^3\left(\sqrt{2\gamma}(b-a)\right)
}
\] which is in agreement with the result reported in \cite{Salminen2007}.
\end{proof}
\begin{theorem}\label{thm:HIHS-contribution}
For $d>0$ and every bounded measurable function $\varphi:\mathbb{R}_+\to\mathbb{R}_+$,
\begin{align}
&\mathbb{E}\!\left(\varphi(M_T^{+});\,M_T^{-}<d,\,H_I<H_S\right)\nonumber\\
&\quad=
\int_0^d
\varphi(r)
\frac{
\left[
\gamma W^{(\gamma)}(r)
-
\bigl(Z^{(\gamma)}(r)-1\bigr)
\dfrac{W^{(\gamma)\prime}(r)}{W^{(\gamma)}(r)}
\right]^2
}{
\gamma W^{(\gamma)}(r)
}
\,dr
\nonumber\\
&\qquad+
\frac{
\left[
\gamma W^{(\gamma)}(d)
-
\bigl(Z^{(\gamma)}(d)-1\bigr)
\dfrac{W^{(\gamma)\prime}(d)}{W^{(\gamma)}(d)}
\right]^2
}{
\gamma W^{(\gamma)}(d)
}
\int_d^{\infty}
\varphi(r)
\exp\left\{
-(r-d)
\frac{W^{(\gamma)\prime}(d)}{W^{(\gamma)}(d)}
\right\}\,dr.
\label{eq:HIHS-contribution}
\end{align}
\end{theorem}
\begin{proof}
On the event $\{H_I<H_S\}$,
\[
M_T^{+}=S_T-I_T,
\qquad
M_T^{-}
=
M_{0,H_I}^{-}\vee M_{H_I,H_S}^{-}\vee M_{H_S,T}^{-}.
\]
Consequently,
\begin{align*}
\Delta
&:=
\mathbb{E}\!\left[
\varphi(M_T^{+});
\,M_T^{-}<d,\ H_I<H_S
\right]\\
&=
\mathbb{E}\!\left[
\begin{gathered}
\varphi(S_T-I_T);\,M_{0,H_I}^{-}<d,\,M_{H_I,H_S}^{-}<d,
M_{H_S,T}^{-}<d,\,H_I<H_S
\end{gathered}
\right].
\end{align*}
Conditioning on $(I_T,S_T,\mathbf{1}_{\{H_I<H_S\}})$ and applying
Theorem~\ref{theorem:mdd-HIHS} yield
\begin{align}
\Delta
={}&
\int_{-\infty}^{0}\int_{0}^{\infty}
\varphi(b-a)f(a,b)\mathbf{1}_{\{d>-a\}}
\nonumber\\
&\quad\times
\left[
f_1(d;a,b)f_2(d;a,b)f_3(d;a,b)
\mathbf{1}_{\{d<b-a\}}
+
\mathbf{1}_{\{d\geq b-a\}}
\right]\,db\,da,
\label{eq:Delta-conditioned}
\end{align}
where $f_1,f_2,f_3$ are the conditional distribution functions in
Theorem~\ref{theorem:mdd-HIHS}, and the joint density of
$(I_T,S_T)$ on $\{H_I<H_S\}$ is
\[
\begin{aligned}
f(a,b)
&=
\frac{W^{(\gamma)\prime}(-a)W^{(\gamma)}(b-a)-W^{(\gamma)}(-a)W^{(\gamma)\prime}(b-a)}
{\bigl(W^{(\gamma)}(b-a)\bigr)^2}\\
&\quad\times
\left[\gamma W^{(\gamma)}(b-a)-\bigl(Z^{(\gamma)}(b-a)-1\bigr)
\dfrac{W^{(\gamma)\prime}(b-a)}{W^{(\gamma)}(b-a)}\right],
\end{aligned}
\]
by Corollary~\ref{corollary_hinf<hsup,inf=a,sup=b}. 

If $d\leq0$, then $\{d>-a\}=\varnothing$ for $a<0$, and hence
$\Delta=0$. Suppose henceforth that $d>0$. Set
\[
x=-a,\qquad r=b-a,
\]
then $0<x<r$, $b=r-x.$ Define
\[
A(x,r):=W^{(\gamma)\prime}(x)W^{(\gamma)}(r)-W^{(\gamma)}(x)W^{(\gamma)\prime}(r).
\]
For $r>d$, direct substitution then gives
\[
f_1(d;-x,r-x)
=
\frac{W^{(\gamma)}(r)}{A(x,r)}
\left(
W^{(\gamma)\prime}(x)-W^{(\gamma)}(x)\frac{W^{(\gamma)\prime}(d)}{W^{(\gamma)}(d)}
\right),
\]
while
\[
\begin{aligned}
f_2(d;-x,r-x)
&=
\frac{W^{(\gamma)}(r)}{W^{(\gamma)}(d)}
\exp\left\{-(r-d)\frac{W^{(\gamma)\prime}(d)}{W^{(\gamma)}(d)}\right\},\\[1mm]
f_3(d;-x,r-x)
&=\frac{\left[\gamma W^{(\gamma)}(d)-\bigl(Z^{(\gamma)}(d)-1\bigr)\dfrac{W^{(\gamma)\prime}(d)}{W^{(\gamma)}(d)}\right]}
{\left[\gamma W^{(\gamma)}(r)-\bigl(Z^{(\gamma)}(r)-1\bigr)\dfrac{W^{(\gamma)\prime}(r)}{W^{(\gamma)}(r)}\right]}.
\end{aligned}
\]
Since
\[
f(-x,r-x)=\frac{A(x,r)}{W^{(\gamma)}(r)^2}\left[\gamma W^{(\gamma)}(r)-\bigl(Z^{(\gamma)}(r)-1\bigr)\dfrac{W^{(\gamma)\prime}(r)}{W^{(\gamma)}(r)}\right],
\]
all terms depending on $A(x,r)$ and $\left[\gamma W^{(\gamma)}(r)-\bigl(Z^{(\gamma)}(r)-1\bigr)\dfrac{W^{(\gamma)\prime}(r)}{W^{(\gamma)}(r)}\right]$ cancel, and therefore
\begin{align}
&f(-x,r-x)
f_1(d;-x,r-x)
f_2(d;-x,r-x)
f_3(d;-x,r-x)
\nonumber\\
&\quad=
\frac{\left[\gamma W^{(\gamma)}(d)-\bigl(Z^{(\gamma)}(d)-1\bigr)\dfrac{W^{(\gamma)\prime}(d)}{W^{(\gamma)}(d)}\right]}{W^{(\gamma)}(d)}\left(
W^{(\gamma)\prime}(x)-W^{(\gamma)}(x)\frac{W^{(\gamma)\prime}(d)}{W^{(\gamma)}(d)}
\right)\exp\left\{
-(r-d)\frac{W^{(\gamma)\prime}(d)}{W^{(\gamma)}(d)}
\right\}.
\label{eq:product-simplification}
\end{align}

The constraint $d>-a$ becomes $0<x<d$. Splitting the integration
domain according as $r\leq d$ or $r>d$, Equation
\eqref{eq:Delta-conditioned} becomes
\begin{align}
\Delta
={}&
\int_0^d\int_x^d
\varphi(r)f(-x,r-x)\,dr\,dx
+
\int_0^d\int_d^\infty
\varphi(r)f(-x,r-x)
\prod_{i=1}^{3}f_i(d;-x,r-x)\,dr\,dx.
\label{eq:Delta-split}
\end{align}

Changing the order of integration in the first term and using
$
Z^{(\gamma)}(r)-1
=
\gamma\int_0^r W^{(\gamma)}(x)\,dx,
$
we have
\begin{align*}
\int_0^r
\bigl(W^{(\gamma)\prime}(x)W^{(\gamma)}(r)-W^{(\gamma)}(x)W^{(\gamma)\prime}(r)\bigr)\,dx
&=
W^{(\gamma)}(r)\bigl(W^{(\gamma)}(r)-W^{(\gamma)}(0)\bigr)
-
\frac{Z^{(\gamma)}(r)-1}{\gamma}W^{(\gamma)\prime}(r)
\\
&=
\frac{W^{(\gamma)}(r)}{\gamma}
\left[\gamma W^{(\gamma)}(r)-\bigl(Z^{(\gamma)}(r)-1\bigr)
\dfrac{W^{(\gamma)\prime}(r)}{W^{(\gamma)}(r)}\right].
\end{align*}
Thus the first term in \eqref{eq:Delta-split} is
\[
\begin{aligned}
\int_0^d
&\varphi(r)
\frac{\left[\gamma W^{(\gamma)}(r)-\bigl(Z^{(\gamma)}(r)-1\bigr)\dfrac{W^{(\gamma)\prime}(r)}{W^{(\gamma)}(r)}\right]}{W^{(\gamma)}(r)}\left(
\frac{\left[\gamma W^{(\gamma)}(r)-\bigl(Z^{(\gamma)}(r)-1\bigr)\dfrac{W^{(\gamma)\prime}(r)}{W^{(\gamma)}(r)}\right]}{\gamma}
\right)\,dr.
\end{aligned}
\]

Similarly, by \eqref{eq:product-simplification},
\begin{align*}
&\int_0^d
\left(
W^{(\gamma)\prime}(x)-W^{(\gamma)}(x)\frac{W^{(\gamma)\prime}(d)}{W^{(\gamma)}(d)}
\right)\,dx
\\
&\quad=
W^{(\gamma)}(d)
-
\frac{Z^{(\gamma)}(d)-1}{\gamma}
\frac{W^{(\gamma)\prime}(d)}{W^{(\gamma)}(d)}
=
\frac{\left[\gamma W^{(\gamma)}(d)-\bigl(Z^{(\gamma)}(d)-1\bigr)\dfrac{W^{(\gamma)\prime}(d)}{W^{(\gamma)}(d)}\right]}{\gamma}.
\end{align*}
It follows that the second term in \eqref{eq:Delta-split} equals
\[
\begin{aligned}
&\frac{\left[\gamma W^{(\gamma)}(d)-\bigl(Z^{(\gamma)}(d)-1\bigr)\dfrac{W^{(\gamma)\prime}(d)}{W^{(\gamma)}(d)}\right]}{W^{(\gamma)}(d)}\left(
\frac{\left[\gamma W^{(\gamma)}(d)-\bigl(Z^{(\gamma)}(d)-1\bigr)\dfrac{W^{(\gamma)\prime}(d)}{W^{(\gamma)}(d)}\right]}{\gamma}
\right)\int_d^\infty
\varphi(r)
\exp\left\{-(r-d)\frac{W^{(\gamma)\prime}(d)}{W^{(\gamma)}(d)}\right\}\,dr.
\end{aligned}
\]
Hence,
\begin{align*}
\Delta
&=
\int_0^d
\varphi(r)
\frac{\left[\gamma W^{(\gamma)}(r)-\bigl(Z^{(\gamma)}(r)-1\bigr)\dfrac{W^{(\gamma)\prime}(r)}{W^{(\gamma)}(r)}\right]^2}
     {\gamma W^{(\gamma)}(r)}\,dr\\
&+
\frac{\left[\gamma W^{(\gamma)}(d)-\bigl(Z^{(\gamma)}(d)-1\bigr)\dfrac{W^{(\gamma)\prime}(d)}{W^{(\gamma)}(d)}\right]^2}
     {\gamma W^{(\gamma)}(d)}
\int_d^\infty
\varphi(r)
\exp\left\{
-(r-d)
\frac{W^{(\gamma)\prime}(d)}
     {W^{(\gamma)}(d)}
\right\}\,dr,
\end{align*}
which proves the statement and substitution of \eqref{eq:brownian-scale-functions} gives
\[
\begin{aligned}
\Delta
&=
-\sqrt{2\gamma}
\frac{\bigl(\cosh(\sqrt{2\gamma} d)-1\bigr)^2}{\sinh^3(\sqrt{2\gamma} d)}
\int_d^\infty \varphi(x)e^{-(x-d)\sqrt{2\gamma}\coth(\kappa d)}\,dx\\
&\quad-\int_0^d \varphi(x)\,
\sqrt{2\gamma}\tanh\!\left(\frac{\sqrt{2\gamma} x}{2}\right)
\frac{\cosh(\sqrt{2\gamma} x)-1}{\sinh^2(\sqrt{2\gamma} x)}
\,dx.
\end{aligned}
\]
which is in agreement with the result given in \cite{Salminen2007} for the standard Brownian motion.
\end{proof}
\begin{corollary}
Let $\alpha,\beta>0.$ Then
\[
\begin{aligned}
&\mathbb{P}\!\left(
M_T^-<\alpha,\,
M_T^+<\beta,\,
H_I<H_S
\right)\\[1mm]
&=
\begin{cases}
\displaystyle
\int_0^\beta
\frac{
\left[
\gamma W^{(\gamma)}(r)
-
\bigl(Z^{(\gamma)}(r)-1\bigr)
\dfrac{W^{(\gamma)\prime}(r)}
{W^{(\gamma)}(r)}
\right]^2
}{
\gamma W^{(\gamma)}(r)
}\,dr,
& 0<\beta\leq\alpha,
\\[6mm]
\displaystyle
\int_0^\alpha
\frac{
\left[
\gamma W^{(\gamma)}(r)
-
\bigl(Z^{(\gamma)}(r)-1\bigr)
\dfrac{W^{(\gamma)\prime}(r)}
{W^{(\gamma)}(r)}
\right]^2
}{
\gamma W^{(\gamma)}(r)
}\,dr
\\[3mm]
\displaystyle\qquad
+
\frac{
\left[
\gamma W^{(\gamma)}(\alpha)
-
\bigl(Z^{(\gamma)}(\alpha)-1\bigr)
\dfrac{W^{(\gamma)\prime}(\alpha)}
{W^{(\gamma)}(\alpha)}
\right]^2
}{
\gamma W^{(\gamma)\prime}(\alpha)
}
\\[3mm]
\displaystyle\qquad\qquad\times
\left[
1-
\exp\!\left\{
-(\beta-\alpha)
\frac{W^{(\gamma)\prime}(\alpha)}
{W^{(\gamma)}(\alpha)}
\right\}
\right],
& \beta>\alpha.
\end{cases}
\end{aligned}
\]
\end{corollary}

\begin{proof}
The assertion follows directly from Theorem~\ref{thm:HIHS-contribution} by taking
\[
\mathbb{P}\!\left(
M_T^-<\alpha,\,
M_T^+<\beta,\,
H_I<H_S
\right)
=
\left.\mathbb{E}\!\left(\varphi(M_T^+);\,M_T^-<d,\,H_I<H_S\right)\right|_{d=\alpha,\ \varphi(x)=\mathbf{1}_{\{x<\beta\}}}.
\]
\end{proof}
\begin{corollary}
Let $\beta>0.$ Then
\[
\mathbb P\!\left(
M_T^+<\beta,\,
H_I<H_S
\right)
=
\int_0^\beta
\frac{
\left[
\gamma W^{(\gamma)}(r)
-
\bigl(Z^{(\gamma)}(r)-1\bigr)
\dfrac{W^{(\gamma)\prime}(r)}
{W^{(\gamma)}(r)}
\right]^2
}{
\gamma W^{(\gamma)}(r)
}\,dr,
\qquad \beta>0.
\]
\end{corollary}

\begin{proof}
Letting $\alpha\to\infty$ in the preceding corollary yields
$
\alpha\wedge\beta \longrightarrow \beta.
$
Moreover, for fixed $\beta>0$, the indicator
$\mathbf 1_{\{\beta>\alpha\}}$ vanishes for all sufficiently large $\alpha$.
Therefore, the second term in the joint expression disappears, and the remaining
integral is taken over $(0,\beta)$.
\end{proof}
\begin{corollary}
Let $\alpha>0$. Then
\[
\begin{aligned}
&\mathbb P\!\left(
M_T^-<\alpha,\,
H_I<H_S
\right)
\\[1mm]
&=
\int_0^\alpha
\frac{
\left[
\gamma W^{(\gamma)}(r)
-
\bigl(Z^{(\gamma)}(r)-1\bigr)
\dfrac{W^{(\gamma)\prime}(r)}
{W^{(\gamma)}(r)}
\right]^2
}{
\gamma W^{(\gamma)}(r)
}\,dr
+
\frac{
\left[
\gamma W^{(\gamma)}(\alpha)
-
\bigl(Z^{(\gamma)}(\alpha)-1\bigr)
\dfrac{W^{(\gamma)\prime}(\alpha)}
{W^{(\gamma)}(\alpha)}
\right]^2
}{
\gamma W^{(\gamma)\prime}(\alpha)
},~\alpha>0.
\end{aligned}
\]
\end{corollary}
\begin{proof}
Letting $\beta\to\infty$ in the joint identity gives
$
\alpha\wedge\beta \longrightarrow \alpha,
~
\mathbf 1_{\{\beta>\alpha\}} \longrightarrow 1.$ Hence the second
term remains, while the exponential factor converges to one. 
\end{proof}

\begin{theorem}\label{thm:HSHI-contribution}
Let $u>0$, and let $\varphi:\mathbb{R}_+\to\mathbb{R}_+$ be bounded and measurable. Then
\begin{align*}
&\mathbb{E}\!\left(\varphi(M_T^-);\,M_T^+<u,\,H_S<H_I\right)\\
&=\frac{\bigl(Z^{(\gamma)}(u)-1\bigr)^2}{\gamma}
\int_u^\infty\varphi(r)
\mathcal L^{-1}_{\theta\to(r-u)}\!\left\{
\frac{\psi(\theta)-\gamma}{Z_\gamma(u,\theta)}\right\}\,dr\\
&\quad+\frac{1}{\gamma}\int_0^u\varphi(r)
\frac{\bigl(Z^{(\gamma)}(r)-1\bigr)^2
\left[\bigl(W^{(\gamma)\prime}(r)\bigr)^2
-W^{(\gamma)}(r)W^{(\gamma)\prime\prime}(r)\right]}
{\bigl(W^{(\gamma)}(r)\bigr)^3}\,dr.
\end{align*}
\end{theorem}

\begin{proof}
On the event $\{H_S<H_I\}$,
\[
M_T^{-}=S_T-I_T,
\qquad
M_T^{+}
=
M_{0,H_S}^{+}
\vee M_{H_S,H_I}^{+}
\vee M_{H_I,T}^{+}.
\]
Therefore, writing
\[
\Theta(u)
:=
\mathbb{E}\!\left[
\varphi(M_T^{-});\,M_T^{+}<u,\ H_S<H_I
\right],
\]
we have
\begin{align*}
\Theta(u)
=
\mathbb{E}\!\left[
\varphi(S_T-I_T);
\,M_{0,H_S}^{+}<u,\,
M_{H_S,H_I}^{+}<u,\,
M_{H_I,T}^{+}<u,\,
H_S<H_I
\right].
\end{align*}
Conditioning on $(I_T,S_T,\mathbf{1}_{\{H_S<H_I\}})$ and applying
Theorem~\ref{theorem:mdu-HSHI} gives
\begin{align}
\Theta(u)
={}&
\int_{-\infty}^{0}\int_0^\infty
\varphi(b-a)g(a,b)\mathbf{1}_{\{u>b\}}
\left[
\prod_{i=1}^{3}g_i(u;a,b)\,
\mathbf{1}_{\{u<b-a\}}
+
\mathbf{1}_{\{u\geq b-a\}}
\right]db\,da,
\label{eq:Theta-conditioned}
\end{align}
where the product uses the established conditional independence of the three path segments and
\[
g(a,b)
=
\frac{
W^{(\gamma)}(-a)\bigl(Z^{(\gamma)}(b-a)-1\bigr)
}{\bigl(W^{(\gamma)}(b-a)\bigr)^3}
\left[
\bigl(W^{(\gamma)\prime}(b-a)\bigr)^2
-
W^{(\gamma)}(b-a)W^{(\gamma)\prime\prime}(b-a)
\right].
\]

For notational convenience, set
\[
A_\gamma(r)
:=
\bigl(W^{(\gamma)\prime}(r)\bigr)^2
-
W^{(\gamma)}(r)W^{(\gamma)\prime\prime}(r)
\]
and recall that $Z^{(\gamma)\prime}(r)=\gamma W^{(\gamma)}(r)$.

If $u\leq0$, the condition $u>b$ is impossible because $b>0$.
Hence $\Theta(u)=0$. Suppose henceforth that $u>0$ and write
$r=b-a$. Substitution of $g$ and $g_1,g_2,g_3$, using in particular the
inverse Laplace transform in Theorem~\ref{theorem:mdu-HSHI}(ii), yields
\begin{align}
\label{g-product}
g(a,b)\prod_{i=1}^{3}g_i(u;a,b)
={}&
W^{(\gamma)}(u-b)
\bigl(Z^{(\gamma)}(u)-1\bigr)
\mathcal L^{-1}_{\theta\to(r-u)}\!\left\{
\frac{\psi(\theta)-\gamma}{Z_\gamma(u,\theta)}
\right\}.
\end{align}

Since $a<0$, we have $r=b-a>b$. Under the transformation
\[
r=b-a,\qquad a=b-r,
\]
the two regions appearing in \eqref{eq:Theta-conditioned} become
\[
\{u>b,\ u<r\}
=
\{0<b<u,\ u<r<\infty\}
\]
and
\[
\{u>b,\ r\leq u\}
=
\{0<b<u,\ b<r\leq u\},
\]
respectively.

Hence, using the expression in \eqref{g-product},
we obtain
\begin{align}
\Theta(u)
={}&
\int_0^u\int_u^\infty
\varphi(r)
W^{(\gamma)}(u-b)
\bigl(Z^{(\gamma)}(u)-1\bigr)
\nonumber\\
&\quad\times
\mathcal L^{-1}_{\theta\to(r-u)}\!\left\{
\frac{\psi(\theta)-\gamma}{Z_\gamma(u,\theta)}
\right\}\,dr\,db
\nonumber\\
&+
\int_0^u\int_b^u
\varphi(r)
\frac{
W^{(\gamma)}(r-b)
\bigl(Z^{(\gamma)}(r)-1\bigr)
A_\gamma(r)
}{\bigl(W^{(\gamma)}(r)\bigr)^3}
\,dr\,db.
\label{eq:Theta-split}
\end{align}

We evaluate the two terms separately.

For the first term, observe that the inverse Laplace
transform depends only on $u$ and $r$, and is therefore
independent of the integration variable $b$.
Consequently, the integration with respect to $b$
can be performed directly.

Using the identity
\[
\int_0^u W^{(\gamma)}(u-b)\,db
=
\int_0^u W^{(\gamma)}(y)\,dy
=
\frac{Z^{(\gamma)}(u)-1}{\gamma},
\]
we obtain
\begin{align}
\Theta_1(u)
={}&
\frac{\bigl(Z^{(\gamma)}(u)-1\bigr)^2}{\gamma}
\int_u^\infty
\varphi(r)
\mathcal L^{-1}_{\theta\to(r-u)}\!\left\{
\frac{\psi(\theta)-\gamma}{Z_\gamma(u,\theta)}
\right\}\,dr.
\label{eq:Theta-first-new}
\end{align}

For the second term, changing the order of integration
gives
\begin{align*}
\Theta_2(u)
={}&
\int_0^u
\varphi(r)
\frac{
\bigl(Z^{(\gamma)}(r)-1\bigr)A_\gamma(r)
}{\bigl(W^{(\gamma)}(r)\bigr)^3}
\left[
\int_0^r W^{(\gamma)}(r-b)\,db
\right]dr.
\end{align*}

Since
\[
\int_0^r W^{(\gamma)}(r-b)\,db
=
\frac{Z^{(\gamma)}(r)-1}{\gamma},
\]
it follows that
\begin{equation}
\Theta_2(u)
=
\frac{1}{\gamma}
\int_0^u
\varphi(r)
\frac{
\bigl(Z^{(\gamma)}(r)-1\bigr)^2 A_\gamma(r)
}{\bigl(W^{(\gamma)}(r)\bigr)^3}\,dr.
\label{eq:Theta-second-new}
\end{equation}

Combining \eqref{eq:Theta-first-new} and
\eqref{eq:Theta-second-new}, we obtain, for $u>0$,
\begin{align*}
&\Theta(u)
=
\frac{\bigl(Z^{(\gamma)}(u)-1\bigr)^2}{\gamma}
\int_u^\infty
\varphi(r)
\mathcal L^{-1}_{\theta\to(r-u)}\!\left\{
\frac{\psi(\theta)-\gamma}{Z_\gamma(u,\theta)}
\right\}\,dr\\
&\quad+
\frac{1}{\gamma}
\int_0^u
\varphi(r)
\frac{
\bigl(Z^{(\gamma)}(r)-1\bigr)^2
\left[
\bigl(W^{(\gamma)\prime}(r)\bigr)^2
-
W^{(\gamma)}(r)W^{(\gamma)\prime\prime}(r)
\right]
}{\bigl(W^{(\gamma)}(r)\bigr)^3}\,dr.
\end{align*}
\end{proof}

\begin{remark}
In the special case where $X$ is a standard Brownian
motion, the preceding identity reduces to the result
obtained in \cite{Salminen2007}.

Indeed, let $k=\sqrt{2\gamma}$. Using
\eqref{eq:brownian-scale-functions}, we have
\[
W^{(\gamma)}(x)=\frac{2}{k}\sinh(kx),
\qquad
Z^{(\gamma)}(x)=\cosh(kx),
\]
and
$
\psi(\theta)=\frac{\theta^2}{2}.
$ Furthermore, from the definition of
$Z_\gamma(u,\theta)$, we obtain
$
Z_\gamma(u,\theta)
=
\cosh(ku)+\frac{\theta}{k}\sinh(ku).
$

Thus, for $t=r-u>0$, direct Laplace inversion yields
\begin{eqnarray*}
&\mathcal L^{-1}_{\theta\to(r-u)}
\left\{
\frac{\psi(\theta)-\gamma}
{Z_\gamma(u,\theta)}
\right\}=
\frac{k^3}{2\sinh^3(ku)}
\exp\left\{
-k(r-u)\coth(ku)
\right\}.
\end{eqnarray*}
Substituting these identities into the
expression for $\Theta(u)$ and simplifying,
we obtain
\begin{align*}
\Theta(u)
&=
\mathbb E\left[
\varphi(S_T-I_T);
M_T^+<u,\ H_S<H_I
\right]\\
&=\sqrt{2\gamma}
\int_0^u
\varphi(r)
\frac{
\left(\cosh(\sqrt{2\gamma}\,r)-1\right)^2
}{
\sinh^3(\sqrt{2\gamma}\,r)
}\,dr\\
&\quad+
\sqrt{2\gamma}
\frac{
\left(\cosh(\sqrt{2\gamma}\,u)-1\right)^2
}{
\sinh^3(\sqrt{2\gamma}\,u)
}
\int_u^\infty
\varphi(r)
\exp\left\{
-(r-u)\sqrt{2\gamma}
\frac{
\cosh(\sqrt{2\gamma}\,u)
}{
\sinh(\sqrt{2\gamma}\,u)
}
\right\}\,dr.
\end{align*}

Thus, the inverse-Laplace-transform representation
reduces exactly to the Brownian motion formula
reported in \cite{Salminen2007}.
\end{remark}
\begin{corollary}
\begin{align}
&\mathbb{P}\left(
M_T^{-}<\alpha,\,
M_T^{+}<\beta,\,
H_S<H_I
\right)
\nonumber\\
&=
\mathbf{1}_{\{\alpha>\beta\}}
\frac{\bigl(Z^{(\gamma)}(\beta)-1\bigr)^2}{\gamma}
\int_{\beta}^{\alpha}
\mathcal L^{-1}_{\theta\to(r-\beta)}
\left\{
\frac{\psi(\theta)-\gamma}{Z_\gamma(\beta,\theta)}
\right\}\,dr
\nonumber\\
&\quad+
\frac{1}{\gamma}
\int_0^{\alpha\wedge\beta}
\frac{
\bigl(Z^{(\gamma)}(r)-1\bigr)^2
\left[
\bigl(W^{(\gamma)\prime}(r)\bigr)^2
-W^{(\gamma)}(r)W^{(\gamma)\prime\prime}(r)
\right]
}{\bigl(W^{(\gamma)}(r)\bigr)^3}\,dr.
\label{eq:HSHI-joint-distribution}
\end{align}
\end{corollary}
\begin{proof}
The assertion follows directly from Theorem~\ref{thm:HSHI-contribution} by taking
\[
\mathbb{P}\!\left(
M_T^-<\alpha,\,
M_T^+<\beta,\,
H_S<H_I
\right)
=
\left.\mathbb{E}\!\left(\varphi(M_T^-);\,M_T^+<u,\,H_S<H_I\right)\right|_{u=\beta,\ \varphi(x)=\mathbf{1}_{\{x<\alpha\}}}.
\]
\end{proof}
\begin{corollary}
Let $\beta>0$. Then
\begin{align}
&\mathbb{P}\left(
M_T^{+}<\beta,\,
H_S<H_I
\right)
\nonumber\\
&=
\frac{
\bigl(Z^{(\gamma)}(\beta)-1\bigr)^2
}{\gamma}
\int_{\beta}^{\infty}
\mathcal L^{-1}_{\theta\to(r-\beta)}
\left\{
\frac{\psi(\theta)-\gamma}
{Z_\gamma(\beta,\theta)}
\right\}\,dr
\nonumber\\
&\quad+
\frac{1}{\gamma}
\int_{0}^{\beta}
\frac{
\bigl(Z^{(\gamma)}(r)-1\bigr)^2
\left[
\bigl(W^{(\gamma)\prime}(r)\bigr)^2
-
W^{(\gamma)}(r)W^{(\gamma)\prime\prime}(r)
\right]
}{
\bigl(W^{(\gamma)}(r)\bigr)^3
}\,dr.
\label{eq:HSHI-drawup-marginal}
\end{align}
\end{corollary}
\begin{proof}
Letting \(\alpha\to\infty\) we have $ \mathbf{1}_{\{\alpha>\beta\}}\longrightarrow 1,$ $\alpha\wedge\beta\longrightarrow\beta.$ Therefore, by continuity from below, \[ \lim_{\alpha\to\infty} \mathbb{P}\left( M_T^{-}<\alpha,\, M_T^{+}<\beta,\, H_S<H_I \right) = \mathbb{P}\left( M_T^{+}<\beta,\, H_S<H_I \right).\] 
\end{proof}
\begin{corollary}
Let $\alpha>0.$ Then
\[
\mathbb{P}\left(
M_T^{-}<\alpha,\,
H_S<H_I
\right)
=
\frac{1}{\gamma}
\int_{0}^{\alpha}
\frac{
\left(Z^{(\gamma)}(r)-1\right)^2
\left[
\left(W^{(\gamma)\prime}(r)\right)^2
-
W^{(\gamma)}(r)W^{(\gamma)\prime\prime}(r)
\right]
}{
\left(W^{(\gamma)}(r)\right)^3
}
\,dr,\quad \alpha>0.
\]
\end{corollary}
\begin{proof}
Letting \(\beta\to\infty\) 
we have
$\alpha\wedge\beta\longrightarrow\alpha.$
Moreover, for fixed \(\alpha>0\),
$\mathbf{1}_{\{\alpha>\beta\}}\longrightarrow 0,$
since \(\beta>\alpha\) for all sufficiently large \(\beta\). Therefore, the first integral term in the joint identity vanishes.
Hence, by continuity from below,
$$
\lim_{\beta\to\infty}
\mathbb{P}\left(
M_T^{-}<\alpha,\,
M_T^{+}<\beta,\,
H_S<H_I
\right)
=
\mathbb{P}\left(
M_T^{-}<\alpha,\,
H_S<H_I
\right).$$\end{proof}
\begin{lemma}\label{lem:inverse-laplace-integral}
Let $\beta>0$ and suppose that $W^{(\gamma)}$ is twice continuously differentiable in a neighborhood of $\beta$. Assume that the inverse Laplace transform below has an integrable restriction to $(0,\infty)$. Then
\begin{align}
&\int_0^\infty \mathcal L^{-1}_{\theta\to t}\left\{\frac{\psi(\theta)-\gamma}{Z_\gamma(\beta,\theta)}\right\}\,dt=\frac{W^{(\gamma)\prime}(\beta)}{\left(W^{(\gamma)}(\beta)\right)^2}-\frac{\gamma}{Z^{(\gamma)}(\beta)}.
\label{eq:inverse-laplace-integral}
\end{align}
Here the inverse Laplace transform is understood on $(0,\infty)$, excluding distributional terms supported at zero.
\end{lemma}
\begin{proof}
By the definition of the generalized scale function,
$$
Z_\gamma(\beta,\theta)=(\psi(\theta)-\gamma)\int_0^\infty e^{-\theta t}W^{(\gamma)}(\beta+t)\,dt,
$$
\[
\frac{\psi(\theta)-\gamma}{Z_\gamma(\beta,\theta)}=\left[\int_0^\infty e^{-\theta t}W^{(\gamma)}(\beta+t)\,dt\right]^{-1}.
\]
As $\theta\to\infty$,
\[
\frac{\psi(\theta)-\gamma}{Z_\gamma(\beta,\theta)}
=\frac{\theta}{W^{(\gamma)}(\beta)}-\frac{W^{(\gamma)\prime}(\beta)}{\left(W^{(\gamma)}(\beta)\right)^2}+O(\theta^{-1}).
\]
The first two terms correspond to distributional contributions supported at $t=0$. Hence the Laplace transform of the restriction of the inverse transform to $(0,\infty)$ is
\[
\frac{\psi(\theta)-\gamma}{Z_\gamma(\beta,\theta)}-\frac{\theta}{W^{(\gamma)}(\beta)}+\frac{W^{(\gamma)\prime}(\beta)}{\left(W^{(\gamma)}(\beta)\right)^2}.
\]
Evaluating at $\theta=0$, using the assumed integrability and $Z_\gamma(\beta,0)=Z^{(\gamma)}(\beta)$, gives \eqref{eq:inverse-laplace-integral}.
\end{proof}

\begin{remark}
Decomposing according to the ordering of the extrema gives
\begin{eqnarray}\label{eq:mdd-ordering-decomposition}
\mathbb P(M_T^-<\beta)&=\mathbb P(M_T^-<\beta,H_S<H_I)+\mathbb P(M_T^-<\beta,H_I<H_S).
\end{eqnarray}

Combining the corresponding results for the two orderings, we obtain, for $\beta>0$,
\[
\mathbb P(M_T^-<\beta)=1-Z^{(\gamma)}(\beta)+\gamma\frac{\left(W^{(\gamma)}(\beta)\right)^2}{W^{(\gamma)\prime}(\beta)},
\]
recovering the standard distribution of maximum drawdown at an independent exponential time given in \cite{Avram2004}.

Similarly,
\begin{eqnarray}
\mathbb P(M_T^+<\beta)&=\mathbb P(M_T^+<\beta,H_S<H_I)+\mathbb P(M_T^+<\beta,H_I<H_S).
\label{eq:mdu-ordering-decomposition}
\end{eqnarray}
Using the marginal result for $H_S<H_I$, Lemma~\ref{lem:inverse-laplace-integral}, and the corresponding result for $H_I<H_S$, we obtain
\begin{align*}
&\mathbb P(M_T^+<\beta)
=\frac{\left(Z^{(\gamma)}(\beta)-1\right)^2}{\gamma}\left[\frac{W^{(\gamma)\prime}(\beta)}{\left(W^{(\gamma)}(\beta)\right)^2}-\frac{\gamma}{Z^{(\gamma)}(\beta)}\right]\\
&\quad+\frac1\gamma\int_0^\beta\frac{\left(Z^{(\gamma)}(r)-1\right)^2\left[\left(W^{(\gamma)\prime}(r)\right)^2-W^{(\gamma)}(r)W^{(\gamma)\prime\prime}(r)\right]}{\left(W^{(\gamma)}(r)\right)^3}\,dr\\
&\quad+\int_0^\beta\frac{\left[\gamma W^{(\gamma)}(r)-\left(Z^{(\gamma)}(r)-1\right)\frac{W^{(\gamma)\prime}(r)}{W^{(\gamma)}(r)}\right]^2}{\gamma W^{(\gamma)}(r)}\,dr.
\end{align*}
Since $Z^{(\gamma)\prime}(r)=\gamma W^{(\gamma)}(r)$, the sum of the last two integrals equals
\[
Z^{(\gamma)}(\beta)-1-\frac{\left(Z^{(\gamma)}(\beta)-1\right)^2W^{(\gamma)\prime}(\beta)}{\gamma\left(W^{(\gamma)}(\beta)\right)^2}.
\]
Consequently,
\begin{align*}
\mathbb P(M_T^+<\beta)
&=Z^{(\gamma)}(\beta)-1-\frac{\left(Z^{(\gamma)}(\beta)-1\right)^2}{Z^{(\gamma)}(\beta)}=1-\frac1{Z^{(\gamma)}(\beta)}.
\end{align*}
Thus, the decomposition recovers the standard distribution of maximum drawup at an independent exponential time given in \cite{Pistorious2004}.
\end{remark}
\begin{theorem}[Joint distribution of maximum drawdown and maximum drawup]
\label{thm:joint-mdd-mdu}
 For any $\alpha,\beta>0$,
\begin{equation}\label{eq:joint-mdd-mdu-compact}
\begin{aligned}
&\mathbb P\left(M_T^-<\alpha,\,M_T^+<\beta\right)\\[2pt]
&=\begin{cases}
\displaystyle
Z^{(\gamma)}(\alpha)-1
-\frac{[Z^{(\gamma)}(\alpha)-1]^2 W^{(\gamma)\prime}(\alpha)}
{\gamma[W^{(\gamma)}(\alpha)]^2}\\[4pt]
\displaystyle\quad+
\frac{\left[\gamma W^{(\gamma)}(\alpha)
-[Z^{(\gamma)}(\alpha)-1]
\frac{W^{(\gamma)\prime}(\alpha)}{W^{(\gamma)}(\alpha)}\right]^2}
{\gamma W^{(\gamma)\prime}(\alpha)}\\[7pt]
\displaystyle\qquad\times\left[1-
\exp\left\{-(\beta-\alpha)
\frac{W^{(\gamma)\prime}(\alpha)}{W^{(\gamma)}(\alpha)}\right\}\right],
& 0<\alpha\le\beta,\\[13pt]
\displaystyle
Z^{(\gamma)}(\beta)-1
-\frac{[Z^{(\gamma)}(\beta)-1]^2 W^{(\gamma)\prime}(\beta)}
{\gamma[W^{(\gamma)}(\beta)]^2}\\[4pt]
\displaystyle\quad+
\frac{[Z^{(\gamma)}(\beta)-1]^2}{\gamma}
\int_\beta^\alpha
\mathcal L^{-1}_{\theta\to(r-\beta)}
\left\{\frac{\psi(\theta)-\gamma}{Z_\gamma(\beta,\theta)}\right\}\,dr,
&0<\beta<\alpha.
\end{cases}
\end{aligned}
\end{equation}
The inverse Laplace transform in the second branch is understood on $r-\beta>0$, as in the intermediate-process maximum-drawup identity.
\end{theorem}

\begin{proof}
The decomposition
\begin{eqnarray*}
&\mathbb P(M_T^-<\alpha,M_T^+<\beta)=\mathbb P(M_T^-<\alpha,M_T^+<\beta,H_I<H_S)+\mathbb P(M_T^-<\alpha,M_T^+<\beta,H_S<H_I)
\end{eqnarray*}
is disjoint. We distinguish the two possible orderings of $\alpha$ and $\beta$.

First, using $Z^{(\gamma)\prime}(r)=\gamma W^{(\gamma)}(r)$, direct differentiation yields
\begin{eqnarray}
&&\frac{d}{dr}\left[
Z^{(\gamma)}(r)-1-
\frac{[Z^{(\gamma)}(r)-1]^2W^{(\gamma)\prime}(r)}
{\gamma[W^{(\gamma)}(r)]^2}\right]\\&=&
\frac{[Z^{(\gamma)}(r)-1]^2
\left([W^{(\gamma)\prime}(r)]^2-
W^{(\gamma)}(r)W^{(\gamma)\prime\prime}(r)\right)}
{\gamma[W^{(\gamma)}(r)]^3}
+
\frac{\left[\gamma W^{(\gamma)}(r)-
[Z^{(\gamma)}(r)-1]
\frac{W^{(\gamma)\prime}(r)}{W^{(\gamma)}(r)}\right]^2}
{\gamma W^{(\gamma)}(r)}.\nonumber
\label{eq:compact-derivative}
\end{eqnarray}
The expression in square brackets on the left tends to zero as $r\downarrow0$ under the assumed unbounded-variation and regularity conditions. Hence the integral of the right-hand side over $(0,x)$ equals
\begin{equation}\label{eq:compact-common}
Z^{(\gamma)}(x)-1-
\frac{[Z^{(\gamma)}(x)-1]^2 W^{(\gamma)\prime}(x)}
{\gamma[W^{(\gamma)}(x)]^2},\qquad x>0.
\end{equation}

If $0<\alpha\le\beta$, then on $\{H_S<H_I\}$ the maximum drawdown equals the total range; the condition $M_T^+<\beta$ is therefore redundant once $M_T^-<\alpha$. The $H_S<H_I$ contribution is the drawdown marginal integrated over $(0,\alpha)$. By the corresponding $H_I<H_S$ joint-distribution result, the other contribution consists of its integral over $(0,\alpha)$ together with the exponential correction
\begin{eqnarray*}
&\frac{\left[\gamma W^{(\gamma)}(\alpha)-
[Z^{(\gamma)}(\alpha)-1]
\frac{W^{(\gamma)\prime}(\alpha)}{W^{(\gamma)}(\alpha)}\right]^2}
{\gamma W^{(\gamma)\prime}(\alpha)}\left[1-\exp\left\{-(\beta-\alpha)
\frac{W^{(\gamma)\prime}(\alpha)}{W^{(\gamma)}(\alpha)}\right\}\right].
\end{eqnarray*}
Adding the two integral contributions and applying \eqref{eq:compact-common} with $x=\alpha$ gives the first branch.

If $0<\beta<\alpha$, then on $\{H_I<H_S\}$ the maximum drawup equals the total range, so the drawdown condition is redundant once $M_T^+<\beta$. Thus the $H_I<H_S$ contribution is its drawup marginal integrated over $(0,\beta)$. By the preceding $H_S<H_I$ joint-distribution corollary, the $H_S<H_I$ contribution is its drawdown-density integral over $(0,\beta)$ plus the additional term
\begin{eqnarray*}
    \frac{[Z^{(\gamma)}(\beta)-1]^2}{\gamma}
\int_\beta^\alpha
\mathcal L^{-1}_{\theta\to(r-\beta)}
\left\{\frac{\psi(\theta)-\gamma}{Z_\gamma(\beta,\theta)}\right\}\,dr.
\end{eqnarray*}
Adding the two integrals over $(0,\beta)$ and applying \eqref{eq:compact-common} with $x=\beta$ yields the second branch. When $\alpha=\beta$, the exponential correction vanishes, and the two branches agree. 
\end{proof}
\section{Moments and correlation of the maximum drawdown and maximum drawup}\label{sec:correlation}
We finally derive the mixed moment and correlation directly from the marginal and joint tail distributions.  In keeping with the preceding sections, no auxiliary kernel notation is introduced: every term is written explicitly through $W^{(\gamma)}$, $Z^{(\gamma)}$, and their derivatives.
\begin{theorem}
\label{thm:correlation-snlp}
Let $X$ be a spectrally negative L\'evy process of unbounded variation and let $T$ be an independent exponential time with parameter $\gamma>0$.
Then
\begin{align}
&\mathbb E[M_T^-M_T^+]\nonumber\\
&=\int_0^\infty x\left[Z^{(\gamma)}(x)-\gamma\frac{[W^{(\gamma)}(x)]^2}{W^{(\gamma)\prime}(x)}\right]dx
+\int_0^\infty\frac{x}{Z^{(\gamma)}(x)}\,dx\nonumber\\
&\quad-\int_0^\infty\frac{\left[\gamma W^{(\gamma)}(x)-[Z^{(\gamma)}(x)-1]\frac{W^{(\gamma)\prime}(x)}{W^{(\gamma)}(x)}\right]^2 W^{(\gamma)}(x)}{\gamma[W^{(\gamma)\prime}(x)]^2}\,dx\nonumber\\
&\quad-\frac1\gamma\int_0^\infty[Z^{(\gamma)}(x)-1]^2\left[\frac1{W^{(\gamma)}(x)}-\frac{\psi'(0+)+\gamma\int_0^x Z^{(\gamma)}(y)\,dy}{[Z^{(\gamma)}(x)]^2}\right]dx.
\label{eq:explicit-mixed-moment}
\end{align}
Moreover, the correlation coefficient is 
\begingroup\small
\begin{align}
&\operatorname{Corr}(M_T^-,M_T^+)\nonumber\\
&=\Bigg[\int_0^\infty x\left(Z^{(\gamma)}(x)-\gamma\frac{[W^{(\gamma)}(x)]^2}{W^{(\gamma)\prime}(x)}\right)dx+\int_0^\infty\frac{x}{Z^{(\gamma)}(x)}\,dx\nonumber\\
&\quad-\int_0^\infty\frac{\left[\gamma W^{(\gamma)}(x)-[Z^{(\gamma)}(x)-1]\frac{W^{(\gamma)\prime}(x)}{W^{(\gamma)}(x)}\right]^2 W^{(\gamma)}(x)}{\gamma[W^{(\gamma)\prime}(x)]^2}\,dx\nonumber\\
&\quad-\frac1\gamma\int_0^\infty[Z^{(\gamma)}(x)-1]^2\left[\frac1{W^{(\gamma)}(x)}-\frac{\psi'(0+)+\gamma\int_0^x Z^{(\gamma)}(y)\,dy}{[Z^{(\gamma)}(x)]^2}\right]dx\nonumber\\
&\quad-\left[\int_0^\infty\left(Z^{(\gamma)}(x)-\gamma\frac{[W^{(\gamma)}(x)]^2}{W^{(\gamma)\prime}(x)}\right)dx\right]\left[\int_0^\infty\frac{dx}{Z^{(\gamma)}(x)}\right]\Bigg]\Bigg[\left\{2\int_0^\infty x\left(Z^{(\gamma)}(x)-\gamma\frac{[W^{(\gamma)}(x)]^2}{W^{(\gamma)\prime}(x)}\right)dx\right.\nonumber\\
&\qquad\left.-\left[\int_0^\infty\left(Z^{(\gamma)}(x)-\gamma\frac{[W^{(\gamma)}(x)]^2}{W^{(\gamma)\prime}(x)}\right)dx\right]^2\right\}\times\left\{2\int_0^\infty\frac{x}{Z^{(\gamma)}(x)}\,dx-\left[\int_0^\infty\frac{dx}{Z^{(\gamma)}(x)}\right]^2\right\}\Bigg]^{-1/2}.
\label{eq:explicit-correlation}
\end{align}
\endgroup
\end{theorem}
\begin{proof}
The marginal survival functions are
\begin{equation}\label{eq:marginal-survival}
\mathbb P(M_T^->x)=Z^{(\gamma)}(x)-\gamma\frac{[W^{(\gamma)}(x)]^2}{W^{(\gamma)\prime}(x)},\qquad
\mathbb P(M_T^+>x)=\frac1{Z^{(\gamma)}(x)}.
\end{equation}
The tail-integral identities give
\begin{align*}
\mathbb E[M_T^-]&=\int_0^\infty\left[Z^{(\gamma)}(x)-\gamma\frac{[W^{(\gamma)}(x)]^2}{W^{(\gamma)\prime}(x)}\right]dx,\\
\mathbb E[(M_T^-)^2]&=2\int_0^\infty x\left[Z^{(\gamma)}(x)-\gamma\frac{[W^{(\gamma)}(x)]^2}{W^{(\gamma)\prime}(x)}\right]dx,\\
\mathbb E[M_T^+]&=\int_0^\infty\frac{dx}{Z^{(\gamma)}(x)},\qquad
\mathbb E[(M_T^+)^2]=2\int_0^\infty\frac{x}{Z^{(\gamma)}(x)}\,dx.
\end{align*}
Also
\begin{align}\label{eq:survival-basic}
\mathbb P(M_T^->\alpha,M_T^+>\beta)=1-\mathbb P(M_T^-<\alpha)-\mathbb P(M_T^+<\beta)+\mathbb P(M_T^-<\alpha,M_T^+<\beta).
\end{align}
We use the two branches of the compact joint-distribution given in Theorem~\ref{thm:joint-mdd-mdu}.

\medskip\noindent\textit{Case 1: $0<\alpha\leq\beta$.}
Substituting the first branch into \eqref{eq:survival-basic}, and simplifying, gives
\begin{align}\label{eq:survival-first}
\mathbb P(M_T^->\alpha,M_T^+>\beta)=\frac1{Z^{(\gamma)}(\beta)}-\frac{\left[\gamma W^{(\gamma)}(\alpha)-[Z^{(\gamma)}(\alpha)-1]\frac{W^{(\gamma)\prime}(\alpha)}{W^{(\gamma)}(\alpha)}\right]^2}{\gamma W^{(\gamma)\prime}(\alpha)}\exp\left\{-(\beta-\alpha)\frac{W^{(\gamma)\prime}(\alpha)}{W^{(\gamma)}(\alpha)}\right\}.
\end{align}
\medskip\noindent\textit{Case 2: $0<\beta<\alpha$.}
The second branch of Theorem~\ref{thm:joint-mdd-mdu}, together with \eqref{eq:survival-basic}, gives
\begin{align}\label{eq:survival-second}
&\mathbb P(M_T^->\alpha,M_T^+>\beta)\nonumber\\
&=Z^{(\gamma)}(\alpha)-\gamma\frac{[W^{(\gamma)}(\alpha)]^2}{W^{(\gamma)\prime}(\alpha)}-\frac{[Z^{(\gamma)}(\beta)-1]^2}{\gamma}\int_{\alpha-\beta}^\infty\mathcal L^{-1}_{\theta\to t}\left\{\frac{\psi(\theta)-\gamma}{Z_\gamma(\beta,\theta)}\right\}dt.
\end{align}
Here the integral over $(0,\alpha-\beta)$ has been replaced by its complementary integral using the marginal maximum-drawup identity
\begin{align*}
&\frac{[Z^{(\gamma)}(\beta)-1]^2}{\gamma}\int_0^\infty\mathcal L^{-1}_{\theta\to t}\left\{\frac{\psi(\theta)-\gamma}{Z_\gamma(\beta,\theta)}\right\}dt\\
&=\frac1{Z^{(\gamma)}(\beta)}-Z^{(\gamma)}(\beta)+1+\frac{[Z^{(\gamma)}(\beta)-1]^2W^{(\gamma)\prime}(\beta)}{\gamma[W^{(\gamma)}(\beta)]^2}.
\end{align*}

Using the joint tail-integral identity,
\begin{equation}\label{eq:mixed-tail}
\mathbb E[M_T^-M_T^+]=\int_0^\infty\int_0^\infty\mathbb P(M_T^->\alpha,M_T^+>\beta)\,d\beta\,d\alpha,
\end{equation}
The two marginal contributions become
\begin{align}\label{eq:marginal-terms}
&\int_0^\infty x\left[Z^{(\gamma)}(x)-\gamma\frac{[W^{(\gamma)}(x)]^2}{W^{(\gamma)\prime}(x)}\right]dx
+\int_0^\infty\frac{x}{Z^{(\gamma)}(x)}\,dx.
\end{align}
For the first correction, arising from $0<\alpha\leq\beta$, use
\[
\int_\alpha^\infty\exp\left\{-(\beta-\alpha)\frac{W^{(\gamma)\prime}(\alpha)}{W^{(\gamma)}(\alpha)}\right\}d\beta
=\frac{W^{(\gamma)}(\alpha)}{W^{(\gamma)\prime}(\alpha)}.
\]
Thus its integral equals
\begin{equation}\label{eq:first-correction}
\int_0^\infty\frac{\left[\gamma W^{(\gamma)}(x)-[Z^{(\gamma)}(x)-1]\frac{W^{(\gamma)\prime}(x)}{W^{(\gamma)}(x)}\right]^2W^{(\gamma)}(x)}{\gamma[W^{(\gamma)\prime}(x)]^2}\,dx.
\end{equation}
For the second correction, arising from $0<\beta<\alpha$, by changing the order of integration we have
\begin{align}\label{eq:second-correction}
&\int_\beta^\infty\int_{\alpha-\beta}^\infty\mathcal L^{-1}_{\theta\to t}\left\{\frac{\psi(\theta)-\gamma}{Z_\gamma(\beta,\theta)}\right\}dt\,d\alpha\\
&=\int_0^\infty\left[\int_\beta^{\beta+t}d\alpha\right]\mathcal L^{-1}_{\theta\to t}\left\{\frac{\psi(\theta)-\gamma}{Z_\gamma(\beta,\theta)}\right\}dt\\
&=\int_0^\infty t\,\mathcal L^{-1}_{\theta\to t}\left\{\frac{\psi(\theta)-\gamma}{Z_\gamma(\beta,\theta)}\right\}dt.
\end{align}
We now evaluate its inner integral explicitly. The generalized scale-function identity gives
\[
\frac{\psi(\theta)-\gamma}{Z_\gamma(\beta,\theta)}=\left[\int_0^\infty e^{-\theta t}W^{(\gamma)}(\beta+t)\,dt\right]^{-1}.
\]
As $\theta\to\infty$, the smoothness assumption yields
\[
\int_0^\infty e^{-\theta t}W^{(\gamma)}(\beta+t)\,dt
=\frac{W^{(\gamma)}(\beta)}{\theta}+\frac{W^{(\gamma)\prime}(\beta)}{\theta^2}+O(\theta^{-3}),
\]
so that
\begin{align}\label{eq:large-theta}
\frac{\psi(\theta)-\gamma}{Z_\gamma(\beta,\theta)}
=\frac{\theta}{W^{(\gamma)}(\beta)}-\frac{W^{(\gamma)\prime}(\beta)}{[W^{(\gamma)}(\beta)]^2}+O(\theta^{-1}).
\end{align}
The first two terms on the right are distributional contributions supported at $t=0$. The Laplace transform of the regular part on $t>0$ is
\[
\frac{\psi(\theta)-\gamma}{Z_\gamma(\beta,\theta)}-\frac{\theta}{W^{(\gamma)}(\beta)}+\frac{W^{(\gamma)\prime}(\beta)}{[W^{(\gamma)}(\beta)]^2}.
\]
Differentiating at $\theta=0$, we find
\begin{align}\label{eq:first-transform-moment}
\int_0^\infty t\,\mathcal L^{-1}_{\theta\to t}\left\{\frac{\psi(\theta)-\gamma}{Z_\gamma(\beta,\theta)}\right\}dt=\frac1{W^{(\gamma)}(\beta)}-\left.\frac{\partial}{\partial\theta}\left[\frac{\psi(\theta)-\gamma}{Z_\gamma(\beta,\theta)}\right]\right|_{\theta=0}.
\end{align}
To calculate the derivative, recall
\[
Z_\gamma(\beta,\theta)=e^{\theta\beta}\left[1+(\gamma-\psi(\theta))\int_0^\beta e^{-\theta y}W^{(\gamma)}(y)\,dy\right].
\]
Differentiation gives
\begin{align}\label{eq:quotient-derivative}
\left.\frac{\partial Z_\gamma(\beta,\theta)}{\partial\theta}\right|_{\theta=0}
&=\beta Z^{(\gamma)}(\beta)-\psi'(0+)\int_0^\beta W^{(\gamma)}(y)\,dy-\gamma\int_0^\beta yW^{(\gamma)}(y)\,dy\nonumber\\
&=\int_0^\beta Z^{(\gamma)}(y)\,dy-\frac{\psi'(0+)}{\gamma}[Z^{(\gamma)}(\beta)-1],
\end{align}
where we used $Z^{(\gamma)\prime}=\gamma W^{(\gamma)}$ and integration by parts. 
Combining \eqref{eq:first-transform-moment} and \eqref{eq:quotient-derivative} yields
\begin{align}\label{eq:inverse-elimination}
&\int_0^\infty t\,\mathcal L^{-1}_{\theta\to t}\left\{\frac{\psi(\theta)-\gamma}{Z_\gamma(\beta,\theta)}\right\}dt=\frac1{W^{(\gamma)}(\beta)}-\frac{\psi'(0+)+\gamma\int_0^\beta Z^{(\gamma)}(y)\,dy}{[Z^{(\gamma)}(\beta)]^2}.
\end{align}
Inserting \eqref{eq:inverse-elimination} into \eqref{eq:second-correction} and combining with \eqref{eq:marginal-terms} and \eqref{eq:first-correction} proves \eqref{eq:explicit-mixed-moment}.

Finally, the covariance is $\mathbb E[M_T^-M_T^+]-\mathbb E[M_T^-]\mathbb E[M_T^+]$, and substituting the displayed tail-integral expressions gives \eqref{eq:explicit-correlation}, completing the proof.
\end{proof}

\begin{remark}For standard Brownian motion with zero drift, substituting \eqref{eq:brownian-scale-functions} into the preceding spectrally negative Levy process moment results give,
\[
\begin{aligned}
\mathbb{E}[M_T^-]
&=\mathbb{E}[M_T^+]
=\frac{\pi}{2\sqrt{2\gamma}},\\
\mathbb{E}[(M_T^-)^2]
&=\mathbb{E}[(M_T^+)^2]
=\frac{2G}{\gamma},\\
\operatorname{Var}(M_T^-)
&=\operatorname{Var}(M_T^+)
=\frac{1}{\gamma}\left(2G-\frac{\pi^2}{8}\right),\\[1mm]
\mathbb{E}[M_T^-M_T^+]
&=\frac{2G-2\log2+1}{\gamma},\\
\operatorname{Cov}(M_T^-,M_T^+)
&=\frac{1}{\gamma}\left(2G-2\log2+1-\frac{\pi^2}{8}\right),\\[1mm]
\operatorname{Corr}(M_T^-,M_T^+)
&=\frac{2G-2\log2+1-\dfrac{\pi^2}{8}}{2G-\dfrac{\pi^2}{8}}
\approx0.3542718536,
\end{aligned}
\]
where $G$ denotes the Catalan's constant.  Note that the correlation coefficient is independent of the exponential rate $\gamma$ and all these results recover the exponential time horizon results acquired in \cite{Salminen2007}. 
\end{remark}

\section{Conclusion}\label{sec:conclusion}

A scale-function framework has been developed for the joint distribution of the maximum drawdown and maximum drawup of a spectrally negative L\'{e}vy process observed at an independent exponential time horizon.  The methodology applied is decomposition of the sample path according to the two possible orderings of the  infimum and  supremum.  Conditional on the  extrema values and on the ordering of their attainment times, the relevant pre-, intermediate and post- components are independent and are characterized by Doob \(h\)-transforms of killed spectrally negative L\'{e}vy processes.  These \(h\)-functions are expressed explicitly through \(W^{(\gamma)}\), \(Z^{(\gamma)}\) and their derivatives.
The resulting conditional laws yield product representations for the maximum drawdown and the maximum drawup on the decomposed path segments.  Integration over the joint laws of the extrema then gives the joint distribution of \((M_T^-,M_T^+)\).  

This study can be extended in various direction for future studies. First, the same decomposition can be combined with occupation-time or excursion identities to include drawdown and drawup durations.  Second, when scale functions are available analytically or numerically, the formulae can be used to compute joint risk measures such as conditional expected maximum drawdown given a maximum drawup, or conditional expected maximum drawup knowing that infimum is attained before the supremum. Third, the exponential time horizon results provide a starting point for numerical inversion or randomized-horizon approximations for deterministic maturities.  These directions would connect the present fluctuation-theoretic results with statistical estimation, insurance and quantitative risk assesments.

\begin{acks}[Funding]  
The authors gratefully acknowledge support from T\"UB\.ITAK 1001 project 124F094. 
\end{acks}

\begin{acks}[Conflict of Interest]
The authors declare that they have no conflict of interest. 
\end{acks}


\begin{thebibliography}{99}

\bibitem[Albrecher and Ivanovs(2017)]{Albrecher2017}
H. Albrecher and J. Ivanovs.
On the joint distribution of tax payments and capital injections for a L\'{e}vy risk model.
\emph{Probability and Mathematical Statistics} 37(2):219--227, 2017.

\bibitem[Asmussen et~al.(2004)]{Asmussen2004}
S. Asmussen, F. Avram and M. Pistorius.
Russian and American put options under exponential phase-type L\'{e}vy models.
\emph{Stochastic Processes and their Applications} 109(1):79--111, 2004.

\bibitem[Avram and Minca(2015)]{Avram2015}
F. Avram and A. Minca.
Steps towards a management toolkit for central branch risk networks, using rational approximations and matrix scale functions.
In \emph{Modern Trends in Controlled Stochastic Processes: Theory and Applications}, 263--278, 2015.

\bibitem{Avram2004}
F. Avram, A. E. Kyprianou and M. R. Pistorius.
Exit problems for spectrally negative L\'{e}vy processes and applications to (canadized) Russian options.
\emph{Annals of Applied Probability} 14(1):215--238, 2004.

\bibitem[Avram et~al.(2008)]{Avram2008}
F. Avram, Z. Palmowski and M. R. Pistorius.
Exit problem of a two-dimensional risk process from the quadrant: exact and asymptotic results.
\emph{Annals of Applied Probability} 18(6):2421--2449, 2008.

\bibitem[Avram et~al.(2019)]{Avram2019}
F. Avram, D. Grahovac and C. Vardar-Acar.
The $W,Z/\nu,\delta$ paradigm for the first passage of strong Markov processes without positive jumps.
\emph{Risks} 7(1), 2019.

\bibitem{avram2020}
F. Avram, D. Grahovac and C. Vardar-Acar.
The $W,Z$ scale functions kit for first passage problems of spectrally negative L\'{e}vy processes, and applications to control problems.
\emph{ESAIM: Probability and Statistics} 24:454--525, 2020.

\bibitem[Azcue and Muler(2005)]{Azcue2005}
P. Azcue and N. Muler.
Optimal reinsurance and dividend distribution policies in the Cram\'{e}r--Lundberg model.
\emph{Mathematical Finance} 15(2):261--308, 2005.

\bibitem{Baurdoux2017}
E. J. Baurdoux, Z. Palmowski, P. Zbigniew and R. Martijn 
On future drawdowns of Lévy processes.
\emph{Stochastic Processes and their Applications} 127(8):2679--2698, 2017.
 
\bibitem[Bertoin(1993)]{Bertoin1993}
J. Bertoin.
Splitting at the infimum and excursions in half-lines for random walks and L\'{e}vy processes.
\emph{Stochastic Processes and their Applications} 47(1):17--35, 1993.

\bibitem{Bertoin1996}
J. Bertoin. \emph{Lévy Processes}. Cambridge University Press, Cambridge Tracts in Mathematics, 1996.
  
\bibitem[Carr et~al.(2011)]{Carr2011}
P. Carr, H. Zhang and O. Hadjiliadis.
Maximum drawdown insurance.
\emph{International Journal of Theoretical and Applied Finance} 14(8):1195--1230, 2011.

\bibitem[Chaumont(1996)]{Chaumont1996}
L. Chaumont.
Conditionings and path decompositions for L\'{e}vy processes.
\emph{Stochastic Processes and their Applications} 64(1):39--54, 1996.

\bibitem[Chaumont and Doney(2005)]{Chaumont2005}
L. Chaumont and R. A. Doney.
On L\'{e}vy processes conditioned to stay positive.
\emph{Electronic Journal of Probability} 10:948--961, 2005.

\bibitem[Chekhlov et~al.(2005)]{Chekhlov2005}
A. Chekhlov, S. Uryasev and M. Zabarankin.
Drawdown measure in portfolio optimization.
\emph{International Journal of Theoretical and Applied Finance} 8(1):13--58, 2005.

\bibitem[Cont and Tankov(2004)]{Cont2004}
R. Cont and P. Tankov.
\emph{Financial Modelling with Jump Processes}.
Chapman and Hall/CRC, 2004.

\bibitem[Cvitanic and Karatzas(1995)]{cvitanic1995portfolio}
J. Cvitanic and I. Karatzas.
On portfolio optimization under drawdown constraints.
\emph{IMA Volume in Mathematical Finance} 65:35--42, 1995.

\bibitem[\c{C}a\u{g}lar et~al.(2022)]{ccauglar2022optimal}
M. \c{C}a\u{g}lar, A. Kyprianou and C. Vardar-Acar.
An optimal stopping problem for spectrally negative Markov additive processes.
\emph{Stochastic Processes and their Applications} 150:1109--1138, 2022.

\bibitem[Douady et~al.(2000)]{Douady2000}
R. Douady, A. N. Shiryaev and M. Yor.
On probability characteristics of downfalls in a standard Brownian motion.
\emph{Theory of Probability and Its Applications} 44(1):29--38, 2000.

\bibitem[Duquesne(2003)]{Duquesne2003}
T. Duquesne.
Path decompositions for real L\'{e}vy processes.
\emph{Annales de l'Institut Henri Poincar\'{e}, Probabilit\'{e}s et Statistiques} 39(2):339--370, 2003.

\bibitem[Grossman and Zhou(1993)]{grossman1993optimal}
S. J. Grossman and Z. Zhou.
Optimal investment strategies for controlling drawdowns.
\emph{Mathematical Finance} 3(3):241--276, 1993.

\bibitem[Hadjiliadis and Ve\v{c}e\v{r}(2006)]{Hadjiliadis2006}
O. Hadjiliadis and J. Ve\v{c}e\v{r}.
Drawdowns preceding rallies in the Brownian motion model.
\emph{Quantitative Finance} 6(5):403--409, 2006.

\bibitem{10.1214/ECP.v20-3945}
Y. Hu, Z. Shi and M. Yor.
The maximal drawdown of the Brownian meander.\emph{Electronic Communications in Probability} 20:1--6, 2015.

\bibitem[Huang et~al.(2013)]{Huang2013}
W. Huang, C. Weng and Y. Zhang.
Multivariate risk models under heavy-tailed risks.
\emph{Applied Stochastic Models in Business and Industry}, 2013.

\bibitem[Ivanovs and Palmowski(2012)]{Ivanovs2012}
J. Ivanovs and Z. Palmowski.
Occupation densities in solving exit problems for Markov additive processes and their reflections.
\emph{Stochastic Processes and their Applications} 122(9):3342--3360, 2012.

\bibitem{rivero2013}
A. Kuznetsov, A. E. Kyprianou and V. Rivero.
The theory of scale functions for spectrally negative L\'{e}vy processes.
In \emph{L\'{e}vy Matters II}, Lecture Notes in Mathematics 2061, 97--186. Springer, 2013.

\bibitem[Kyprianou(2006)]{kyprianou2006introductory}
A. E. Kyprianou.
\emph{Introductory Lectures on Fluctuations of L\'{e}vy Processes with Applications}.
Springer, 2006.

\bibitem[Kyprianou(2014)]{Kyprianou2014}
A. E. Kyprianou.
\emph{Fluctuations of L\'{e}vy Processes with Applications: Introductory Lectures}.
Second edition. Springer, 2014.

\bibitem[Kyprianou et~al.(2014)]{Kyprianou2014a}
A. E. Kyprianou, J. C. Pardo and J. L. P\'{e}rez.
Occupation times of refracted L\'{e}vy processes.
\emph{Journal of Theoretical Probability} 27(4):1292--1315, 2014.

\bibitem[Landriault et~al.(2015)]{Landriault2015}
D. Landriault, B. Li and S. Li.
Analysis of a drawdown-based regime-switching L\'{e}vy insurance model.
\emph{Insurance: Mathematics and Economics} 60:98--107, 2015.

\bibitem[Landriault et~al.(2016)]{Landriault2016}
D. Landriault, B. Li and S. Li.
Drawdown analysis for the renewal insurance risk process.
\emph{Scandinavian Actuarial Journal}, 1--19, 2016.

\bibitem[Landriault et~al.(2017)]{Landriault2017}
D. Landriault, B. Li and H. Zhang.
On magnitude, asymptotics and duration of drawdowns for L\'{e}vy models.
\emph{Bernoulli} 23(1):432--458, 2017.

\bibitem[Leal and Mendes(2005)]{Leal2005}
R. P. C. Leal and B. V. M. Mendes.
Maximum drawdown.
\emph{Journal of Alternative Investments} 7(4):83--91, 2005.

\bibitem[Li et~al.(2019)]{Li2019}
B. Li, N. L. Vu and X. Zhou.
Exit problems for general drawdown times of spectrally negative L\'{e}vy processes.
\emph{Journal of Applied Probability} 56(2):441--457, 2019.

\bibitem[Magdon-Ismail et~al.(2004)]{Magdon-Ismail2004}
M. Magdon-Ismail, A. F. Atiya, A. Pratap and Y. S. Abu-Mostafa.
On the maximum drawdown of a Brownian motion.
\emph{Journal of Applied Probability} 41(1):147--161, 2004.

\bibitem{risks7040105}
E. Mayerhofer.
Three Essays on Stopping.
\emph{Risks} 7(4):105, 2019.

\bibitem[Mijatovi\'{c} and Pistorius(2012)]{Mijatovic2012}
A. Mijatovi\'{c} and M. R. Pistorius.
On the drawdown of completely asymmetric L\'{e}vy processes.
\emph{Stochastic Processes and their Applications} 122(11):3812--3836, 2012.

\bibitem[Millar(1977)]{Millar1977}
P. W. Millar.
Zero-one laws and the minimum of a Markov process.
\emph{Transactions of the American Mathematical Society} 226:365--391, 1977.

\bibitem[{\O}ksendal(2014)]{oksendal2014stochastic}
B. {\O}ksendal.
\emph{Stochastic Differential Equations: An Introduction with Applications}.
Springer, 2014.


\bibitem[Palmowski and Tumilewicz(2018)]{Palmowski2018}
Z. Palmowski and J. Tumilewicz.
Pricing insurance drawdown-type contracts with underlying L\'{e}vy assets.
\emph{Insurance: Mathematics and Economics} 79:1--14, 2018.

\bibitem[Papapantoleon(2008)]{papapantoleon2008introduction}
A. Papapantoleon.
An introduction to L\'{e}vy processes with applications in finance.
Lecture notes, 2008.

\bibitem[P\'{e}rez et~al.(2018)]{Perez2018}
J.-L. P\'{e}rez, K. Yamazaki and A. Bensoussan.
Optimal periodic replenishment policies for spectrally positive L\'{e}vy demand processes.
arXiv:1806.09216, 2018.

\bibitem[Pistorius(2004)]{Pistorious2004}
M.R. Pistorius. On exit and ergodicity of the spectrally one-sided 
L´evy process reflected at its infimum.
\emph{J. Theor. Probab.} 17, 183–220 (2004)

\bibitem[Pospisil and Ve\v{c}e\v{r}(2009)]{Pospisil2009}
L. Pospisil and J. Ve\v{c}e\v{r}.
PDE methods for the maximum drawdown.
\emph{Journal of Computational Finance} 12(2):63--77, 2009.

\bibitem[Pospisil and Ve\v{c}e\v{r}(2010)]{Pospisil2010}
L. Pospisil and J. Ve\v{c}e\v{r}.
Portfolio sensitivity to changes in the maximum and the maximum drawdown.
\emph{Quantitative Finance} 10(6):617--627, 2010.

\bibitem[Rossello and Lo Cascio(2021)]{Rossello2021}
D. Rossello and S. Lo Cascio.
A refined measure of conditional maximum drawdown.
\emph{Risk Management} 23(4):301--321, 2021.

\bibitem[Salminen and Vallois(2007)]{Salminen2007}
P. Salminen and P. Vallois.
On maximum increase and decrease of Brownian motion.
\emph{Annales de l'Institut Henri Poincar\'{e} (B) Probability and Statistics} 43:655--676, 2007.

\bibitem[Salminen and Vallois(2020)]{Salminen2020}
P. Salminen and P. Vallois.
On the maximum increase and decrease of one-dimensional diffusions.
\emph{Stochastic Processes and their Applications} 130(9):5592--5604, 2020.

\bibitem{salminen2025drawdowns}
 P. Salminen and P. Vallois.
  Drawdowns of diffusions.
\emph{ESAIM: Probability and Statistics} 29:357--380, 2025.
  

\bibitem[Schuhmacher and Eling(2011)]{schuhmacher2011sufficient}
F. Schuhmacher and M. Eling.
Sufficient conditions for expected utility to imply drawdown-based performance rankings.
\emph{Journal of Banking and Finance} 35(9):2311--2318, 2011.

\bibitem[Shepp and Shiryaev(1993)]{shepp1993russian}
L. Shepp and A. N. Shiryaev.
The Russian option: reduced regret.
\emph{Annals of Applied Probability} 3(3):631--640, 1993.

\bibitem[Shreve(2004)]{shreve2004stochastic}
S. E. Shreve.
\emph{Stochastic Calculus for Finance II: Continuous-Time Models}.
Springer, 2004.

\bibitem[Sornette(2003)]{Sornette2003}
D. Sornette.
\emph{Why Stock Markets Crash: Critical Events in Complex Financial Systems}.
Princeton University Press, 2003.

\bibitem[Vardar-Acar et~al.(2021)]{VardarAcar2021}
C. Vardar-Acar, M. \c{C}a\u{g}lar and F. Avram.
Maximum drawdown and drawdown duration of spectrally negative L\'{e}vy processes decomposed at extremes.
\emph{Journal of Theoretical Probability} 34:1486--1505, 2021.

\bibitem[Vardar-Acar and \c{C}a\u{g}lar(2017)]{VardarAcar2017}
C. Vardar-Acar and M. \c{C}a\u{g}lar.
Maximum loss and maximum gain of spectrally negative L\'{e}vy processes.
\emph{Extremes} 20(2):301--308, 2017.

\bibitem [Ceren Vardar Acar and Mine Çağlar(2026)]{VardarAcar2026}
      C. Vardar-Acar and M. \c{C}a\u{g}lar. Conditional Path Decomposition at the Infimum and Maximum Drawdowns for Spectrally Negative L\'{e}vy Processes,
      \emph{arXiv} eprint 2606.27573, 2026
      
\bibitem[Ve\v{c}e\v{r}(2006)]{Vecer2006}
J. Ve\v{c}e\v{r}.
Maximum drawdown and directional trading.
Frankfurt Math Finance Workshop Paper, 2006.

\bibitem[Ve\v{c}e\v{r}(2007)]{Vecer2007}
J. Ve\v{c}e\v{r}.
Preventing portfolio losses by hedging maximum drawdown.
\emph{Wilmott} 5(4):1--8, 2007.

\bibitem[Zhang and Hadjiliadis(2010)]{Zhang2010}
H. Zhang and O. Hadjiliadis.
Drawdowns and rallies in a finite time horizon.
\emph{Methodology and Computing in Applied Probability} 12(2):293--308, 2010.

\bibitem[Zhang and Hadjiliadis(2012)]{Zhang2012}
H. Zhang and O. Hadjiliadis.
Drawdowns and the speed of market crash.
\emph{Methodology and Computing in Applied Probability} 14(3):739--752, 2012.

\bibitem{ZHANG2023104669}
G. Zhang and L. Li.
A general method for analysis and valuation of drawdown risk.
\emph{Journal of Economic Dynamics and Control} 152:104--669,2023.

\bibitem[Cont and Tankov(2004)]{contfinancial}
R. Cont and P. Tankov.
\emph{Financial Modelling with Jump Processes}.
Chapman and Hall/CRC, 2004.

\end{thebibliography}
\end{document}